\documentclass[12pt]{article}
\usepackage{setspace}
\usepackage{amsmath,amssymb,amsfonts,amsthm}
\newenvironment{myabstract}{\par\noindent
{\bf Abstract . } \small }
{\par\vskip8pt minus3pt\rm}
\newcounter{item}[section]
\newcounter{kirshr}
\newcounter{kirsha}
\newcounter{kirshb}
\newenvironment{enumarab}{\setcounter{kirshb}{1}
\begin{list}{(\arabic{kirshb})}{\usecounter{kirshb}} }{\end{list}}
\newtheorem{theorem}{Theorem}[section]

\newtheorem{question}[theorem]{Question}
\newtheorem{lemma}[theorem]{Lemma}
\newtheorem{corollary}[theorem]{Corollary}
\newenvironment{demo}[1]{\noindent{\bf #1.}\upshape\mdseries}
{\nopagebreak{\hfill\rule{2mm}{2mm}\nopagebreak}\par\normalfont}
\theoremstyle{definition}

\newtheorem{definition}[theorem]{Definition}

\def\C{{\mathfrak{C}}}
\def\Fm{{\mathfrak{Fm}}}

\def\At{{\sf At}}

\def\Fr{{\mathfrak{Fr}}}
\def\Sg{{\mathfrak{Sg}}}
\def\Fm{{\mathfrak{Fm}}}
\def\A{{\mathfrak{A}}}
\def\B{{\mathfrak{B}}}
\def\C{{\mathfrak{C}}}
\def\D{{\mathfrak{D}}}
\def\M{{\mathfrak{M}}}

\def\Ig{{\mathfrak{Ig}}}

\def\Rd{{\mathfrak{Rd}}}
\def\(R)RA{{\bf (R)RA}}

\def\c #1{{\cal #1}}

\def\B{{\sf B}}

 \def\Cm{{\mathfrak{Cm}}}

\def\Rl{{\mathfrak{Rl}}}

\def\A{{\mathfrak{A}}}
\def\B{{\mathfrak{B}}}
\def\C{{\mathfrak{C}}}
\def\D{{\mathfrak{D}}}

\def\L{{\mathfrak{L}}}

\def\Bl{\mathfrak{Bl}}

\def\At{{\mathfrak{At}}}

\def\c #1{{\cal #1}}

\def\A{{\cal{A}}}
\def\B{{\mathfrak{B}}}
\def\C{{\mathfrak{C}}}
\def\D{{\mathfrak{D}}}
\def\P{{\mathfrak{P}}}

\def\F{{\mathfrak{F}}}

\def\c#1{{\mathcal #1}}

\def\Cm{{\mathfrak Cm}}
\def\Sg{{\mathfrak Sg}}
\def\A{{\mathfrak A}}
\def\F{{\mathfrak F}}

\def\At{{\sf At}}

\title{On the multi dimensional modal logic of substitutions}
\author{Tarek Sayed Ahmed and  Mohammad Assem}

\begin{document}
\maketitle

\begin{myabstract} We prove completeness, interpolation,
decidability and an omitting types theorem for certain multi dimensional modal logics where the states
are not abstract entities but have an inner structure. The states will be sequences.
Our approach is algebraic addressing (varieties generated by)
complex algebras of Kripke semantics for such logic. Those algebras, whose elements are sets of states
are common reducts of cylindric and polyadic
algebras.

\footnote{Mathematics Subject Classification. 03G15; 06E25

Key words: multimodal logic, substitution algebras, interpolation}
\end{myabstract}

\section{Introduction}

We study certain propositional multi dimensional modal logics.
These logics arise from the modal view of certain algebras that are cylindrifier free reducts of both cylindric algebras and
polyadic algebras of dimension $n$, $n\geq 2$.
We can view the class of representable such algebras from the perspective of (multi dimensional) modal
logic as complex algebras of certain relational structures. This approach was initiated by
Venema, by viewing quantifiers which are the most prominent citizens of first order logic as modalities.

Modal logic can be studied as fragments of first order logic. But one can turn the glass around
and examine first order logic as if it were a modal formalism. The basic idea enabling this perspective is that we may view
assignments the functions that give first order variables their value in a first order structure as states of a modal
model, and this makes standard first order connectives behave just
like modal diamonds and boxes. First order logic then forms an example of a multi dimensional modal system.
Multi dimensional modal logic is a branch of modal logic dealing with special relational structures
in which the states, rather than being abstract entities, have some inner structure.
These states are tuples over some base set, the domain of the first order structure.


S\'agi studied $n$ dimensional relational structures, referred to also as Kripke frames, of the form
$(V,  S_j^i)_{i,j\in n}$ where $V\subseteq {}^nU$ and for $x,y\in V$, $(x, y)\in S_j^i$, iff $x\circ [i/j]=y$. Here $[i/j]$ is the replacement
that sends $i$ to $j$ and is the identity map otherwise.
In this paper we study Kripke frames of the form $(V, S_{[i,j]})$ where
where $[i,j]$ is the transposition that swaps $i$ and $j$, $V\subseteq {}^nU$ and for $x,y\in V$, $(x, y)\in { S}_{[i,j]}$, iff $x\circ [i,j]=y$.
By dealing with such cases, the case when we have both kinds of substitutions corresponding to
replacements and transpositions cannot help but come to mind.
We shall show that this case, is also very interesting, and gives a sort of unifying framework for our investigations.
A Kripke frame is called square if $V={}^nU$ for some set $U$. The complex algebras of such frames are called full set algebras.
Subdirect products of such set algebras is called the class of representable algebras of dimension $n$.


S\'agi, answering a question posed by Andr\'eka,
proves that in the case of replacements, the class of representable algebras is a finitely axiomatizable quasi variety, that is {\it not} a variety,
and that the variety generated by this class (closing it under homomorphic images) is obtained by relativizing units to so-called
diagonalizable sets; $V$ is diagonalizable if whenever $s\in V$,
then $s\circ [i/j]\in V$. He also provides finite axiomatizations for both the quasivariety (by quasi-equations) and the generated variety (by equations).
This provides a strong completeness theorem for such modal logics:
There is a fairly simple finite Hilbert style proof system,
that captures all the valid multimodal formulas with the above Kripke semantics, when the assignments are only restricted to
diagonalizable sets. In other words, this modal logic is complete with respect to diagonalizable frames,
but not relative to square frames.

The idea of finding such an axiomatization is not too hard; it basically involves translating finite presentations
of the semigroup generated by replacements
to equations,  and further stipulating that the substitution operators on algebra are Boolean endomorphisms.

We show that in the case when we have, transpositions and replacements, the subdirect products
of complex algebras of square frames,
is a finitely axiomatizable variety,  while in the case when only transpositions are available,  this class is only a quasivariety, that is not
a variety. In all cases we obtain axiomatizations for the varieties in question (closing our quasivariety under homomorphic images in case of transpositions
only), by translating finite
presentations of the semigroup $^nn$, and the symmetric group $S_n.$

We can also view such logics as a natural fragment of $L_n$ the first order logic restricted to $n$ variables, that is strong enough to
describe certain combinatorial
properties.  Here we have finitely many variables, namely exactly $n$ variables, atomic formulas are of the form
$R(x_0,\ldots x_{n-1})$ with the variables occurring in their natural order.
The substitution operators are now viewed as unary connectives. In this case,
it is more appropriate to deal with square frames, representing ordinary models, and substitutions
have their usual meaning of substituting variables for free variables such that the substitution is free.

From the perspective of algebraic logic, we are dealing with a non-trivial instance of the so called representation problem, which roughly says
the following. Given a class $K$ of concrete algebras, like set algebras, can we find a simple, hopefully finite axiomatization by equations of $V(K)$,
the variety generated by $K$. The representation problem for Boolean algebras is completely settled by Stones theorem; every abstract Boolaen algebra is isomorphic to a
set algebra. The representation problem, for cylindric-like algebras is more subtle; indeed not every cylindric algebra is representable,
and any axiomatization of the variety generated by the set algebras of dimension $>2$ is highly complex.
A recent negative extremely strong result, proved by Hodkinson for cylindric algebras, is that
the it is undecidable whether a finite cylindric algebra is representable.
The representation problem have provoked extensive research,
and, indeed, it is still an active part of  algebraic logic.

In this paper we obtain a clean cut solution to the representation problem for several natural classes of set algebras, whose elements are sets of sets
endowed with concrete set theoretic operations, those of substitutions, and possibly diagonal elements.

We also discuss the infinite version of such logics.
Algebraically we deal with infinite dimensional algebras; and from the modal point of view we deal with infinitely many
modalities. In case of the presence of both transpositions and replacements, we show that like the finite dimensional case,
subdirect product of full set algebras forms a variety.

The dual facet of our logics (as multi dimensional modal logics and fragments of first order logic) enables us to go further in the analysis
and give two proofs that such logics have the Craig interpolation property
and enjoy other definability properties
like Beth definability and Robinson's joint consistency theorem. This works for all dimensions.

Our results readily follows from the fact that, in all cases, the variety generated by the set algebras has the superamalgamation property.
Our first proof, inspired by the 'first order view' is a Henkin construction. In this case, we show that the free algebras have the interpolation property.
The second proof inspired by the 'modal view' uses known
correspondence theorems between Kripke frames and complex algebras, and closure properties that Kripke frames should satisfy.

In the case of transpositions, for finite as well as for infinite dimensions, we
prove an omitting types theorem for countable languages using the Baire Category theorem.
We also show in this case that atomic algebras posses complete representations,
i.e representations preserving infinitary meets
carrying them to set-theoretic intersection.

Indeed, there are various types of representations in algebraic logic. Ordinary representations are just isomorphisms from Boolean algebras
with operators to a more concrete structure (having the same similarity type)
whose elements are sets endowed with set-theoretic operations like intersection and complementation.
Complete representations, on the other hand, are representations that preserve arbitrary conjunctions whenever defined.
The notion of complete representations has turned out to be very interesting for cylindric algebras, where it is proved
by Hirsch and Hodkinson that the class of completely representable algebras is not elementary.

The correlation of atomicity to complete representations has caused a lot of confusion in the past.
It was mistakenly thought for a long time, among algebraic logicians,  that atomic representable relation and cylindric algebras
are completely representable, an error attributed to Lyndon and now referred to as Lyndon's error. Here we show, that atomic algebras
{\it are} completely representable.

Another property that is important in algebraic logic generally, and particulary in cylindric and polyadic algebras is the question whether
the Boolean reducts of the free algebras are atomic or not. N\'emeti has shown that for cylindric algebras, this corresponds to a form of
Godel's incompleteness
theorem. More precisely, N\'emeti showed that finite variable fragments of first order logic enjoy a Godel's incompleteness theorem;
there are formulas that cannot be extended to complete recursive theories;
in the formula (free) algebras such formulas cannot be atoms, further they cannot be isolated from the bottom element by atoms.
In short the free algebras are not atomic.
Here we show, in sharp contrast to the cylindric case,  that finitely generated  algebras are finite, hence atomic. In particular,
our varieties are locally finite.

We learn from our results that subdirect products of full transposition set algebras, is only a quasivariety
(it is not closed under taking homomorphic images, theorem  \ref{not}), but the algebraic closure of this quasivariety
share a lot of properties with Boolean algebras; it is finitely axiomatizable,
theorem \ref{rep}, it is locally finite, theorem \ref{locallyfinite}, atomic algebras are completely representable, theorem \ref{com},
and it has the superamalgamation property, theorem \ref{SUPAP}. (this occurs in all dimensions).

When we have both transpositions and replacements, our quasivariety turns out to be a variety, theorems \ref{variety}, \ref{v2},
but we lose complete  representability of atomic algebras, or rather we do not know whether atomic algebras are completely representable.
In this case, we could only show that the canonical extension of any algebra is always completely representable,
and that the minimal completion of a completely representable algebra,
stays completely representable, theorems \ref{canonical}, \ref{completion}. Superamalgamation and local finiteness hold, theorems \ref{SUPAP},
\ref{locallyfinite} in this case as well. An oddity that occurs here, is that we could provide a finite schema axiomatizing such varieties,
endowed with diagonal elements
 in the infinite dimensional case, theorem \ref{infinite}, but we did not succeed to provide a finite axiomatization
for finite dimensions (with diagonal elements).

Finally, we prove that the validity problem for our logics in all dimensions is decidable.

\subsection*{Summary}

In section 2 we fix our notation
and recall the basic concepts.

In section 3 we study transposition algebras
obtaining
a finite axiomatization of the quasi-variety (that is shown not to be closed under taking homomorphic images)
of representable algebras and proving strong representability results.

In section 3 we deal with algebras when we have transpositions and replacements together. In this case
we show that the class of subdirect products of algebras (in all dimensions) is a variety and we provide
a simple equational axiomatization this of variety.

In section 4 we deal with free algebras, and we formulate and prove some general theorems adressing free algebras of classes of Boolean algebras
with operators.

In section 5, we also formulate and prove general theorems on the amalgamation property, and
we prove that the varieties (generated by the quasi-varieties in case of replacements only and transpositions only)
have the superamalgamation theorem.

Logical counterpart of our algebraic theorems are carefully formulated in the final section.

\section{Notation and Preliminaries}

Our system of notation is mostly standard,  or self-explainatory, but the following list may be useful. Throughout, for every natural number $n$
we follow Von-Neuman's convention; $n=\{0,\ldots,n-1\}.$ Let $A$ and $B$ be sets. Then ${}^AB$ denotes the set of functions whose domain is
$A$ and whose range is a subset of $B.$ In addition,  $A^*$ denotes the set of finite sequences over $A$ and
$\mathcal{P}(A)$ denotes the power set of $A$, that is, the set of all subsets of $A$. For a cardinal $\kappa$, we let $A\subseteq_\kappa B$ mean that $A$ is a subset of $B$ with cardinality $|A|=\kappa.$


If $f:A\longrightarrow B$ is a function and $X\subseteq A,$ then $f\upharpoonright X$ is the restriction of $f$ to $X$.

If $K$ is a class of algebras, then $\mathbf{H}K,\;\mathbf{P}K,\;\mathbf{S}K,\;\mathbf{Up}K$
denote the classes of homomorphic images, (isomorphic copies of)
direct products, (isomorphic copies of) subalgebras, (isomorphic copies of) ultraproducts of members of $K$,
respectively. If $\Sigma$ is a set of formulas, then
$\mathbf{Mod}(\Sigma)$ denotes the class of all models of $\Sigma.$

By a \emph{quasi-equation} we mean a universal formula of the form $e_0\wedge\ldots\wedge e_n\Rightarrow f,$ where $e_0,\ldots, e_n,$
and $f$ are equations. Let $K$ be a class of algebras. Then,

$K$ is a \emph{variety} iff $K$ is axiomatizable by equations iff $\mathbf{HSP}K=K,$ and
$K$ is a \emph{quasi-variety} iff $K$ is axiomatizable by quasi-equations iff $\mathbf{SPUp}K=K$.

We distinguish notationally between an algebra and its universe. Algebras are denoted by Gothic letters,
and when we write, for example $\A$, for an algebra, we will be tacitly assuming that the corresponding
Roman letter $A$ is its universe.

The word \emph{BAO} will abbreviate \emph{Boolean algebra with operators} (an algebra that has a Boolean reduct and every non-Boolean
operation on it is additive in each of its arguments). For $K\subseteq BAO$ and a set $X$ and $\Fr_XK$ denotes the algebra freely generated by $X$.
We sometimes identify $\Fr_XK$ with $\Fr_{|X|}K$, and we say that $\Fr_XK$ is the $K$ free algebra on the set of generators $X$, or simply on $X$.
We do not assume that $\Fr_XK$, belongs to $K$, when $K$ is not a variety.
If $K$ is a class of algebras, then $Fin(K)$ is the class of finite algebras in $K$ and $V(K)$ denotes ${\bf HSP}(K)$, that is, the variety generated by $K$.
If $\A$ is a $BAO$, and $X\subseteq \A$, then $\Bl\A$, denotes its Boolean reduct, $\At\A$ denotes the set of atoms of $\Bl\A$, $\Sg^{\A}X$,
and $\Ig^{\A}X$ denote the subalgebra, ideal,
generated by $X$.

Every class $K$ of $BAO$'s correspond to a multi dimensional modal logic. For an algebra $\A$, $\A_+$ denotes its ultrafilter atom structure,
and $\A^+$ denotes $\Cm\A_{+}$ the canonical extension of $\A$. If $\A$ is atomic then $\At\A$ denotes its atom structure,
and $\Cm\At\A$ is its minimal
completion.

The next theorem is folklore:
\begin{theorem}
Suppose $\A$ is a BAO and for every $0^\A\neq a\in \A$ there is a homomorphism $h_a:\A\longrightarrow\mathcal{B}_a$
such that $h_a(a)\neq 0^{\mathcal{B}_a}.$ Then $\A$ is embeddable in the direct product $\prod_{0^\A\neq a\in A}\mathcal{B}_a,$
or equivalently, $\A\in\mathbf{SP}\langle\mathcal{B}_a, 0^\A\neq a\in A\rangle.$
\end{theorem}
\begin{proof}
Define $g:\A\to \prod_{a\in A}\B_a$ by
$g(x)=(h_a(x): a\in A)$.
Then $g$ is an embedding.
\end{proof}

\section{Transposition Algebras}
In this section our work closely resembles that of S\'agi's, but we obtain stronger representability results. We refer to the algebras studied by
S\'agi as replacement algebras.
We not only prove the representability of abstract algebra defined via a finite set of equations, but we further show that such
representations can be chosen to respect infinitary meets and joins. In particular, we show that atomic algebras are completely representable.

We also show that the subdirect products of full set algebras, those algebras whose units are squares is not a variety.

Our treatment of the infinite dimensional case is also different than S\'agi's for replacement algebras;  we alter the units of representable algebras, dealing
with a different generating class (in this case we show that it suffices to take only one set algebra, that is not the free algebra on $\omega$ generators,
as a generating set).


\begin{definition}[Transposition Set Algebras]
Let $U$ be a set. \emph{The full transposition set algebra of dimension} $\alpha$ \emph{with base}
$U$ is the algebra $$\langle\mathcal{P}({}^\alpha U); \cap,-,S_{ij}\rangle_{i\neq j\in\alpha};$$
where  $S_{ij}$'s are unary operations defined
by $$S_{ij}(X)=\{q\in {}^\alpha U:q\circ [i,j]\in X\}.$$ Recall that $[i,j]$ denotes that transposition of $\alpha$ that permutes $i,j$
and leaves any other element fixed. The class of Transposition Set Algebras of dimension $\alpha$ is defined as follows:
$$SetTA_\alpha=\mathbf{S}\{\A:\A\text{ is a full transposition set algebra of dimension }\alpha $$$$ \text{ with base }U,\text{ for some set }U\}.$$
\end{definition}

\begin{definition}[Representable Transposition Set Algebras]
The class of Transposition Set Algebras of dimension $\alpha$ is defined to be $$RTA_\alpha=\mathbf{SP}SetTA_\alpha.$$
\end{definition}

\begin{definition}[Permutable Set]
Let $U$ be a given set, and let $D\subseteq{}^\alpha U.$ We say that $D$ is permutable iff
$$(\forall i\neq j\in\alpha)(\forall s\in{}^\alpha U)(s\in D\Longrightarrow s\circ [i,j]\in D).$$
\end{definition}

\begin{definition}[Permutable Algebras]
The class of \emph{Permutable Algebras} of dimension $\alpha$, $\alpha$ an ordinal,  is defined to be
$$PTA_n=\mathbf{SP}\{\langle\mathcal{P}(D); \cap,-,S_{ij}\rangle_{i\neq j\in \alpha}: U\text{ \emph{is a set}}, D\subset{}^{\alpha}
U\text{\emph{permutable}}\}.$$
Here  $S_{ij}(X)=\{q\in D:q\circ [i,j]\in X\}$, and $-$ is complement w.r.t. $D$.\\
If $D$ is a permutable set then the algebra
$\wp(D)$ is defined to be $$\wp(D)=\langle\mathcal{P}(D);\cap,-,S_{ij}\rangle_{i\neq j\in\alpha}$$ So $\wp(D)\in PTA_\alpha.$
\end{definition}
Note that $\wp(^{\alpha}U)$ can be viewed as the complex algebra of the atom structure $(^{\alpha}U, S_{ij})_{i,j\in \alpha}$ where
for all $i,j, S_{ij}$ is a binary relation, such that for $s,t\in {}^{\alpha}U$, $(s,t)\in S_{ij}$ iff $s\circ [i,j]=t.$ When we consider permutable sets
then from the modal point of view we are restricting the states or assignments to $D$. This process is also referred to as
relativization.

for some time to come we restrict ourselves to finite $\alpha$, which we denote by $n$.
\begin{theorem}
Let $U$ be a set and suppose $G\subseteq{}^n U$ is permutable.
Let $\A=\langle\mathcal{P}({}^n U);\cap,-,S_{ij}\rangle_{i\neq j\in n}$ and let $\mathcal{B}=\langle\mathcal{P}(G);\cap,-,S_{ij}\rangle_{i\neq j\in n}$.
Then the following $h:\A\longrightarrow\mathcal{B}$ defined by $h(x)=x\cap G$ is a homomorphism.

\end{theorem}
\begin{proof}
It is easy to see that $h$ preserves $\cap,-$ so it remains to show that the $S_{ij}$'s are also preserved.
To do this let $i\neq j\in n$ and $x\in A.$ Now\begin{align*}
h({S_{ij}}^\A x)= {S_{ij}}^\A x \cap G &=\{q\in{}^n U:q\circ [i,j]\in x\}\cap G\\
&=\{q\in G:q\circ [i,j]\in x\}\\
&=\{q\in G:q\circ [i,j]\in x\cap G\}\\
&=\{q\in G:q\circ [i,j]\in h(x) \}\\
&={S_{ij}}^\mathcal{B} h(x)
\end{align*}
\end{proof}
The function $h$ will be called \emph{relativization by} $G$.
Now we distinguish certain elements of $SetTA_n$ ($n>1$) which play an important role.

\begin{definition}[Small algebras]
For any natural number $k\leq n$
the algebra $\A_{nk}$ is defined to be $$\A_{nk}=\langle\mathcal{P}({}^nk);\cap,-,S_{ij}\rangle_{i\neq j\in n}.$$ So $\A_{nk}\in SetTA_n$.
\end{definition}

\begin{theorem}
$RTA_n=\mathbf{SP}\{\A_{nk}:k\leq n\}.$
\end{theorem}
\begin{proof}
The proof is exactly like that of Theorem 4.9 in\cite{sagiphd}.
Clearly, $\{\A_{nk}:k\leq n\}\subseteq RTA_n,$ and since, by definition,
$RTA_n$ is closed under the formation of subalgebras and direct products, $RTA_n\supseteq\mathbf{SP}\{\A_{nk}:k\leq n\}.$

To prove the other inclusion, it is enough to show $SetTA_n\subseteq \mathbf{SP}\{\A_{nk}:k\leq n\}.$
Let $\A\in SetTA_n$ and suppose that $U$ is the base of $\A.$
If $U$ is empty, then $\A$ has one element, and one can easily show $\A\cong\A_{n0}.$
Otherwise for every $0^\A\neq a\in A$ we can construct a homomorphism $h_a$ such that
$h_a(a)\neq 0$ as follows. If $a\neq 0^\A$ then there is a sequence $q\in a.$ Let  $U_0^a=range(q)$.
Clearly, $^nU_0^a$ is permutable, therefore by Theorem 3.5 relativizing by $^nU_0^a$ is a homomorphism to
$\A_{nk_a}$ (where $k_a:=|range(q)|\leq n$). Let $h_a$ be this homomorphism. Since $q\in ^nU_0^a$ we have $h_a(a)\neq0^{\A_{nk_a}}.$
Applying Theorem 2.1 one concludes that $\A\in\mathbf{SP}\{\A_{nk}:k\leq n\}$ as desired.
\end{proof}


\subsection{Axiomatizing the Variety Generated by $RTA_n$}

We know that the variety generated by $RTA_n$ is finitely axiomatizable since it is generated by finitely many finite algebras,
and because, having a Boolean reduct, it is congruence distributive. This follows from a famous theorem by Baker.

Throughout, given a set $U$, we let $S_U$ denote the set of permutations on $U$. In this section we will provide a finite set of equations axiomatizing $\mathbf{HSP}RTA_n.$ To do this, we first note that
the set of transpositions $\{[i,j]:i\neq j\in n\}$, with composition of maps, generates the group $S_n$ of permutations of $n$
To obtain our desired axiomatization
we need to recall some concepts from the presentation theory of groups. In particular, we need a concrete
presentation of the Symmetric Group $S_n$. Let $\sigma_i=[i-1,i]$, then $S_n$ is generated by $\sigma_1,\ldots,\sigma_{n-1}$
governed by the following relations:
\begin{enumerate}
\item $\sigma_i^2=1$
\item $\sigma_i\sigma_j=\sigma_j\sigma_i$ for $i\neq j\pm1$
\item $\sigma_i\sigma_{i+1}\sigma_i=\sigma_{i+1}\sigma_i\sigma_{i+1}$
\end{enumerate}
Throughout, $s$ will denote the operation symbol corresponding to $S$ in the language level.
\begin{definition}[The axiomatization]\label{ax}
For all natural numbers $n\in\omega,\;n>1,$ let $\Sigma_n$ be the following finite set of equations.
\begin{enumerate}
\item The usual Boolean algebra axioms.
\item $s_{ij}(x\wedge y)=s_{ij}(x)\wedge s_{ij}(y)$.
\item $s_{ij}(-x)=-s_{ij}(x)$.
\item $s_{ij}s_{ij}x=x.$
\item $s_{i\; i+1}s_{j\;j+1}x=s_{j\;j+1}s_{i\;i+1}x.$
\item $s_{i\;i+1}s_{j\;j+1}s_{i\;i+1}x=s_{j\;j+1}s_{i\;i+1}s_{j\;j+1}x$
\end{enumerate}
\end{definition}

Let $TA_n$ be the abstractly defined class $\mathbf{Mod}(\Sigma_n)$.
Our goal in this section is to show that $TA_n=\mathbf{HSP}RTA_n$.

\begin{definition}
Given a set $U$, let $T(U)=\{s_{ij}:i\neq j\in U\}$ and let $\hat{}:T(U)^*\longrightarrow S_U$
which maps
each $q=(s_{i_1j_1}\ldots s_{i_kj_k})\in T(U)^*$ to $$(s_{i_1j_1}\ldots s_{i_kj_k})^{\hat{}}= [i_1,j_1]\circ\ldots\circ[i_k,j_k]\in S_U.$$
When $q$ is the empty word, $q\hat{}=Id_U.$
\end{definition}

\begin{theorem}
For all $n\in\omega$ the set of (all instances of the) axiom-schemas 4,5,6 is a presentation of the group $S_n$ via generators $T(n)$.
That is, for all $t_1,t_2\in T(n)^*$ we have $$axioms\;4,5,6\vdash t_1=t_2\text{  iff  }t_1^{\hat{}}=t_2^{\hat{}}.$$
Here $\vdash$ denotes derivability using Birkhoff's calculus for equational logic.
\end{theorem}
\begin{proof}
This is clear because $\Sigma_n$ corresponds exactly to the set of relations governing the generators of $S_n$.
\end{proof}

\begin{definition}For every $\xi\in S_n$ we associate a sequence $s_\xi\in T(n)^*$
such that $s_\xi^{\hat{}}=\xi.$ Such an $s_\xi$ exists, since the transpositions of $n$ generate $S_n.$
In other words, in any model, the term $s_\xi (x)$ has the same interpretation as the term $s_{i_1j_1}(\ldots (s_{i_kj_k}(x)))$ if $\xi=[i_1,j_1]\circ\ldots\circ[i_k,j_k].$

In our algebras ($\A=\langle \mathcal{P}(V),\cap,-,S_{ij}\rangle_{i,j\in\alpha}$ with permutable $V$),
$s_\xi$ corresponds to the unary operator $S_\xi$ defined as follows. For $X\subset V,$ $$S^\A_\xi(X)=\{p\in V:p\circ \xi\in X\}.$$
\end{definition}

Now we turn to prove that $ PTA_n\subseteq \mathbf{HSP}RTA_n\subseteq TA_n\subseteq PTA_n$
which achieves our goal.

\begin{theorem}
$PTA_n\subseteq\mathbf{HSP}RTA_n$
\end{theorem}
\begin{proof}
It is enough to show that $\A=\langle\mathcal{P}(G);\cap,-,S_{ij}\rangle_{i\neq j\in n}\in \mathbf{HSP}RTA_n$ whenever $G$ permutable. Let $U=\bigcup_{q\in G}range(q)$ and let $$\mathcal{B}=\langle\mathcal{P}({}^nU);\cap,-,S_{ij}\rangle_{i\neq j\in n}.$$
Clearly, $\mathcal{B}\in SetTA_n$ and so $\A$ is a homomorphic image of $\mathcal{B}$. Thus, $\A\in\mathbf{H}SetTA_n\subseteq \mathbf{HSP}RTA_n$
\end{proof}

\begin{theorem}
$\mathbf{HSP}RTA_n\subseteq TA_n,$ or equivalently $RTA_n\models\Sigma_n.$
\end{theorem}
\begin{proof}It is enough to prove that $SetTA_n\models\Sigma_n$ which is a routine computation.
\end{proof}
The proof of the following lemma is completely analogous to that of Theorem 4.17 in \cite{sagiphd}, so it is omitted.
\begin{lemma}
Let $\A$ be an $RTA_n$ type $BAO$.
Suppose $G\subset {}^nn$ is a permutable set, and $\langle\mathcal{F}_\xi:\xi\in G\rangle$ is a system of ultrafilters of $\A$
such that for all $\xi\in G,\;i\neq j\in n$ and $a\in\A$ the following condition holds:
$${S_{ij}}^\A(a)\in\mathcal{F}_\xi\Leftrightarrow a\in \mathcal{F}_{\xi\circ[i,j]}\quad\quad (*).$$
Then the following function $h:\A\longrightarrow\wp(G)$ is a homomorphism$$h(a)=\{\xi\in G:a\in \mathcal{F}_\xi\}.$$
\end{lemma}

\begin{theorem}\label{rep}
$TA_n=\mathbf{Mod}(\Sigma_n)\subseteq PTA_n.$
\end{theorem}
\begin{proof}
Let $\A\in TA_n$ and let $0^\A\neq a\in A$.
We construct a homomorphism $h$ $$h:\A\longrightarrow\wp(S_n)$$ such that $h(a)\neq 0^{\wp(S_n)}.$
Let us choose an ultrafilter $\mathcal{F}\subseteq A$ containing $a$. Let $h:\A\longrightarrow \wp(S_n)$ be the following function
$h(z)=\{\xi\in S_n:S_{\xi}^\A(z)\in\mathcal{F}\}.$
Now, if $\mathcal{F}_\xi=\{z\in A: S_\xi(z)\in\mathcal{F}\},$ then $\langle\mathcal{F}_\xi:\xi\in S_n\rangle$ is a system of ultrafilters of
$\A$ satisfying $(*)$. To see this, let $i\neq j\in n,\;z\in A$ and $\xi\in S_n.$
Suppose $S_{ij}^\A(z)\in\mathcal{F}_\xi.$ This implies $$S_\xi^\A S_{ij}^\A(z)\in \mathcal{F}\quad\quad(**).$$
Observe that $(s_\xi s_{ij})^{\hat{}}=\xi\circ[i,j].$ Therefore, by Theorem 2.11,
$$Axioms\;4,5,6\vdash s_\xi s_{ij}(x)=s_{\xi\circ[i,j]}(x)\quad (***)$$So by $(**)$ we have $S^\A_{\xi\circ[i,j]}(z)\in\mathcal{F}.$
But then $z\in\mathcal{F}_{\xi\circ[i,j]}.$\\ Conversely, if $z\in\mathcal{F}_{\xi\circ[i,j]}$
then $S^\A_{\xi\circ[i,j]}(z)\in\mathcal{F}$ so by $(***)$ $S_\xi^\A S_{ij}^\A(z)\in \mathcal{F}$ hence $S_{ij}^\A(z)\in\mathcal{F}_{\xi}.$
Now, by previous lemma, $h$ is the desired homomorphism.
\end{proof}
\begin{theorem}\label{not} For $n\geq 2$, $RTA_n$ is not a variety.
\end{theorem}
\begin{proof}
Let us denote $\sigma$ the quasi-equation
$$s_f(x)=-x\longrightarrow 0=1,$$ where $f$ is a permutation that we will define shortly.
It is easy to see that for all $k\leq n,$ the small algebra $\A_{nk}$ (
or more generally, any set algebra with square unit) models $\sigma.$
This can be seen using a constant map in $^nk.$ More precisely, let $q\in  {}^nk$
be an arbitrary constant map and let $X$ be any subset of $^nk.$
We have two cases for $q$ which are $q\in X$ or $q\in -X$. In either case,
noticing that $q\in X\Leftrightarrow q\in S_f(X),$ it cannot be the case that $S_f(X)=-X.$
Thus, the implication $\sigma$ holds in $\A_{nk}.$
It follows then, from Theorem 3.7,
that $RTA_n\models\sigma$ (because the operators $\mathbf{S}$ and $\mathbf{P}$ preserve quasi-equations).

Now we are going to show that there is some element $\B\in PTA_n$ such that $\B\nvDash\sigma.$
Let $G\subseteq {}^nn$ be the following permutable set $$G=\{s\in {}^n2:|\{i:s(i)=0\}|=1\}.$$
Let $\B=\wp(G)$, then $\wp(G)\in PTA_n.$ Let $f$ be the permutation defined as follows
\[
 f =
  \begin{cases}
   [0,1]\circ[2,3]\circ\ldots\circ[n-2,n-1] & \text{if $n$ is even}, \\
   [0,1]\circ[2,3]\circ\ldots\circ[n-3,n-2]      & \text{if $n$ is odd}
  \end{cases}
\]
Notice that $f$ is the composition of disjoint transpositions.
Let $X$ be the following subset of $G,$ $$X=\{e_i:i\mbox{ is odd, }i<n\},$$
where $e_i$ denotes the map that maps every element to $1$ except that the $i$th elementis mapped to $0$.
It is easy to see that, for all odd $i<n,$ $e_i\circ f=e_{i-1}.$
This clearly implies that $$S_f^\B(X)=-X=\{e_i:i\mbox{ is even, }i<n\}.$$
Since $0^\B\neq 1^\B,$ $X$ falsifies $\sigma$ in $\B.$ Since $\B\in {\bf H}{\wp(^nn)}$ we are done.
\end{proof}

\subsection{Complete representability}

In this section, we consider the following question. If $\A$ is a transposition algebra,  is there a
representation of $\A$ that preserves a given set of infinite meets and joins? Is there one, perhaps, which represents all existing meets and joins?
First we investigate the case when homomorphisms respect only a given set of meets, not necessarily all.
For a substitution algebra $\A$, and $a\in A$, $N_a$ denotes the set of all Boolean ultrafilters of $\A$ containing $a$.
Recall that $\{N_a:a\in A\}$ is a clopen base for the Stone topology whose underlying set consists of all ultrafilters of
$\A$.
In what follows $\prod$ and $\sum$ denote infimum
and supremum, respectively Throughout this section $n$ will be a finite ordinal $>1$.
We start by a crucial easy lemma.
Recall that $TA_n=Mod(\Sigma_n)$.

\begin{lemma} Let $\A\in TA_n$ and let $i,j\in n$, then $s_{[i,j]}$ is a complete Boolean endomorphism
\end{lemma}
\begin{proof}
Let $X$ be a subset of $A$. Since $s_{[i,j]}$ is a Boolean endomorphism,
and $\prod X\leq x$ for all $x\in X,$ so $s_{[i,j]}\prod X\leq s_{[i,j]}x$ for all $x\in X$. Therefore, $s_{[i,j]}\prod X\leq \prod s_{[i,j]}X.$
Conversely, $$s_{[i,j]}\prod X\leq \prod s_{[i,j]}X$$ implies that $$\prod X\leq s_{[i,j]}\prod s_{[i,j]}X.$$
But from what we have already done, $$s_{[i,j]}\prod s_{[i,j]}X\leq\prod  s_{[i,j]}s_{[i,j]}X= \prod X.$$
Thus, $$\prod X= s_{[i,j]}\prod s_{[i,j]}X,$$ which implies that $$s_{[i,j]}\prod X=\prod s_{[i,j]}X.$$
\end{proof}

\begin{theorem}\label{OTT} Let $\A\in TA_n$ be countable, let $a\in A$ non-zero, and let $X\subseteq A$
be such that such that $\prod X=0$.
Then there exists permutable $V$ and a representation $h:\A\to \wp(V)$ such that $\bigcap_{x\in X}h(x)=\emptyset$ and $h(a)\neq\phi$.
\end{theorem}
\begin{demo}{Proof} Each $\eta\in S_n$  is a composition of transpositions, so that $s_{\eta}$, a composition of complete endomorphisms,
is itself complete. Therefore $\prod s_{\eta}X=0$ for all $\eta\in S_n$.
Then for all $\eta\in S_n$, $B_{\eta}=\bigcap_{x\in X} N_{s_{\eta}}x$ is nowhere dense in the Stone topology
and $B=\bigcup_{\eta\in S_n} B_{\eta}$ is of the first category (In fact, $B$ is also nowhere dense).
Let $F$ be an ultrafilter that contains $a$ and is outside $B$ which exists by the Baire category theorem,
 since the complement of $B$ is dense. Then for all $\eta\in S_n$,
there exists $x\in X$ such that ${s}_{\tau}x\notin F$. Let $h:\A\to \wp(S_n)$ be the usual representation function; $h(x)=\{\eta\in S_n: { s}_{\eta}x\in F\}$.
Then clearly $\bigcap_{x\in X} h(x)=\emptyset.$
\end{demo}
We refer to $X$ as a non-principal types, and to $V$ as a model omitting $X$.
The proof can be easily generalized to countably many non-principal types.

We now ask for the preservation of possibly uncountably many meets.
We will see that we are actually touching upon somewhat deep issues in set theory here.
We let $MA$ denote Martin's axiom.
But first a piece of notation: For a cardinal $\kappa$,  $OTT(\kappa)$ stands for the above statement when we have $\kappa$ many meets.
\begin{theorem}\label{t}
\begin{enumerate}
\item The statement $``(\forall \kappa<{2}^{\omega})$$OTT(\kappa)"$ is provable in $ZFC$ +$MA$.
\item The statement $``(\forall k< 2^{\omega})(OTT(\kappa))"$ is independent of $ZFC+\neg CH$.
\end{enumerate}
\end{theorem}
\begin{demo}{Proof} \begin{enumerate}\item Martin's axiom implies that, in the Stone space, the union of $\kappa$ many $(\kappa<{2}^{\omega})$ nowhere dense sets is actually a countable union
and so the Baire category theorem readily applies.
\item We have proved consistency since $MA$ implies the required statement.
We now prove independence.
Let $covK$ be the least cardinal $\kappa$ such that the real line can be covered by $\kappa$ many closed disjoint
nowhere dense sets. It is known that
$\omega<covK\leq 2^{\omega}$. In any Polish space the intersection of $< covK$ dense sets is dense.
But then if $\kappa<covK$, then $OTT(\kappa)$ is true.
The independence is proved using standard iterated forcing to show that it is consistent that $covK$ could be literally anything greater
than  $\omega$ and $\leq 2^{\omega}$,
and then show that $OTT(covK)$ is false.
\end{enumerate}
\end{demo}

\begin{corollary} Let $\A\in TA_n$ be countable and $X\subseteq A$ be such that $\prod X=0$. Then there is a permutable set $V$ and
and an embedding $f:\A\to \wp(V)$ such that $\bigcap_{x\in X}f(x)=\emptyset.$
\end{corollary}
\begin{demo}{Proof} For each $a\neq 0$, let $f_a:\A\to \wp(S_n)$
be a homomorphism such that $\bigcap_{x\in X}f_a(x)=\emptyset$. For each $a\in A$, let
$V_a=S_n$ and let $V$ be the disjoint union of the $V_a$'s.
Then $\prod_{a\in A} \wp(V_a)\cong \wp(V)$, via $(a_i:i\in I)\mapsto  \bigcup a_i$. Let $g$ denote this isomorphism.
Define $f:\A\to \wp(V)$ by $f(x)=g[(f_ax: a\in A)]$.Then $f$ is the desired
embedding. Indeed, if $s\in \bigcap_{x\in X} f(x)$, then $s\in V_a$ for some $a\in A$, and $s\in \bigcap_{x\in X}f_a(x)$, and this cannot
happen.
\end{demo}
Now we turn to the problem of preserving all meets. To make the problem more tangible we need a few preparations.
For some tome to come we restrict the notion of representation.
We stipulate that a representation of an algebra $\A$, is a one to one homomorphism  $f:\A\longrightarrow \wp(V)$ for a permutable set $V$.
Notice that $\A$ could be infinite, and so $V$ could be infinite as well.
Let $\A$ be a substitution  algebra and $f:\A\longrightarrow \wp(V)$ be a representation of $\A$.
If $s\in V$, we let
$$f^{-1}(s)=\{a\in \A: s\in f(a)\}.$$
An atomic representation $f:\A\to \wp(V)$ is a representation such that for each
$s\in V$, the ultrafilter $f^{-1}(s)$ is principal.
A complete representation of $\A$ is a representation $f$ satisfying
$$f(\prod X)=\bigcap f[X]$$
whenever $X\subseteq \A$ and $\prod X$ is defined.

When we ask for representations that respect all existing meets, it turns out that the Boolean reduct of
algebras in question have to be atomic in the first place.
Atomicity a necessary condition for complete representability may not sufficient, as is the case example for cylindric algebras.
For transposition algebras we show that complete representability and atomicity (of the Boolean reduct) are equivalent.
To prove this, we first recall the following result established by Hirsch and Hodkinson for cylindric algebras.
The proof works verbatim for $TA$'s.
\begin{lemma}Let $\A\in TA_{n}$.
\begin{enumarab}

\item A representation of $\A$
is atomic if and only if it is complete.
\item A representation $f:\A\to \wp(V)$ is atomic if and only if  $\Bl\A$ is atomic  and $\bigcup_{x\in X}f(x)=V$, where $X$ is the set of atoms.
\item If $\A$ has a complete representation $f$, then $\Bl\A$ is atomic.
\end{enumarab}
\end{lemma}
\begin{proof}See \cite{Hirsh}.
\end{proof}
We say that an algebra is atomic, if its Boolean reduct is atomic.
Contrary to the cylindric case, we have:
\begin{theorem}\label{com} If $\A\in TA_n$ is atomic, then $\A$ ise completely representable
\end{theorem}
\begin{proof} Let $\B$ be an atomic transposition algebra, let $X$ be the set of atoms, and
let $c\in \B$ be non-zero. Let $S$ be the Stone space of $\B$, whose underlying set consists of all Boolean ulltrafilters of
$\B$. Let $X^*$ be the set of principal ultrafilters of $\B$ (those generated by the atoms).
These are isolated points in the Stone topology, and they form a dense set in the Stone topology since $\B$ is atomic.
So we have $X^*\cap T=\emptyset$ for every nowhere dense set $T$ (since principal ultrafilters, which are isolated points in the Stone topology,
lie outside nowhere dense sets).
Recall that for $a\in \B$, $N_a$ denotes the set of all Boolean ultrafilters containing $a$.
Now  for all $\tau\in S_n$, we have
$G_{X, \tau}=S\sim \bigcup_{x\in X}N_{s_{\tau}x}$
is nowhere dense. Let $F$ be a principal ultrafilter of $S$ containing $c$.
This is possible since $\B$ is atomic, so there is an atom $x$ below $c$; just take the
ultrafilter generated by $x$. Also $F$ lies outside the $G_{X,\tau}$'s, for all $\tau\in S_n$
Define, as we did before,  $f_c$ by $f_c(b)=\{\tau\in S_n: s_{\tau}b\in F\}$.
Then clearly for every $\tau\in S_n$ there exists an atom $x$ such that $\tau\in f_c(x)$.
As before, for each $a\in A$, let
$V_a=S_n$ and let $V$ be the disjoint union of the $V_a$'s.
Then $\prod_{a\in A} \wp(V_a)\cong \wp(V)$. Define $f:\A\to \wp(V)$ by $f(x)=g[(f_ax: a\in A)]$.
Then $f: \A\to \wp(V)$ is an embedding such that
$\bigcup_{x\in \At\A}f(x)=V$. Hence $f$ is a complete representation.

\end{proof}
A classical theorem of Vaught for first order logic says that countable atomic theories have countable atomic models,
such models are necessarily prime, and a prime model omits all non principal types.
We have a similar situation here:

\begin{theorem} Let $f:\A\to \wp(V)$ be an atomic representation of $\A\in TA_n$.
Then for any $Y\subseteq A$, if $\prod Y=0$, then $\bigcap_{y\in Y} f(y)=\emptyset$.
\end{theorem}
\begin{proof}Follows from the simple observation that
$\bigcup_{x\in \At\A}f(x)\leq \bigcup_{y\in y} f(-y).$
\end{proof}
Seemingly a second order condition, in contrast to cylindric algebras, we get
\begin{corollary}
The class of completely representable algebras is elementary, and is axiomatized by a finite set of first order sentences
\end{corollary}

\begin{demo}{Proof} Atomicity can be expressed by a first order sentence.
\end{demo}

\subsection{The infinite dimensional case}

S\'agi dealt with infinite dimensional algebras; but he only dealt with square units.
We give a reasonable generalization to the above theorems for the infinite dimensional case, by allowing weak sets as units,
a weak set being a set of sequences that agree
cofinitely with some fixed sequence.
That is a weak set is one of the form $\{s\in {}^{\alpha}U: |\{i\in \alpha, s_i\neq p_i\}|<\omega\}$,
where $U$ is a set, $\alpha$ an ordinal and $p\in {}^{\alpha}U$.
This set will be denoted by $^{\alpha}U^{(p)}$. The set $U$ is called the base of the weak set. A set $V\subseteq {}^{\alpha}\alpha^{(Id)}$, is defined to
be permutable just like the finite dimensional case. Altering top elements to be weak sets, rather than squares,
turns out to be a  fruitful approach and a rewarding task.

\begin{definition}
We let  $PTA_{\alpha}$ be the variety generated by
$$\wp(V)=\langle\mathcal{P}(V),\cap,-,S_{ij}\rangle_{i,j\in\alpha},$$
where $V\subseteq {}^\alpha\alpha^{Id}$ is permutable.

\end{definition}
Let $\Sigma_{\alpha}$ be the set of finite schemas obtained from
$\Sigma_n$ but now allowing indices from $\alpha$. Obviously $\Sigma_{\alpha}$ is infinite.
$(Mod\Sigma_{\alpha}: \alpha\geq \omega)$ is a system of varieties definable by schemes which means that it is enough to specify $\Sigma_{\omega}$,
to define $\Sigma_{\alpha}$ for all $\alpha\geq \omega$.

Indeed, let $\rho:\alpha\to \beta$ be an injection. One defines for a term $t$ in $L_{\alpha}$ a term $\rho(t)$ in $L_{\beta}$ by recursion
$\rho(v_i)=v_i$ and $\rho(f(\tau))=f(\rho(\tau))$.
Then one defines $\rho(\sigma=\tau)=\rho(\sigma)=\rho(\tau)$.
Then there exists a finite set $\Sigma\subseteq \Sigma_{\omega}$ such that
$\Sigma_{\alpha}=\{\rho(e): \rho:\omega\to \alpha \text { is an injection }, e\in \Sigma\}.$

We give two proofs of the following main representation theorem but first we give a definition.
Let $\alpha\leq\beta$ be ordinals and let $\rho:\alpha\rightarrow\beta$ be an injection.
For any $\beta$-dimensional algebra $\B$)
we define an $\alpha$-dimensional algebra $\Rd^\rho(\c B)$, with the same base and Boolean structure as
$\c B$, where the $(i,j)$th substitution  of $\Rd^\rho(\c B)$ is $s_ {\rho(i)\rho(j)}\in\c B$.
For a class $K$, $\Rd^{\rho}K=\{\Rd^{\rho}\A: \A\in K\}$. When $\alpha\subseteq \beta$ and $\rho$ is the identity map on $\alpha$, then we write
$\Rd_{\alpha}\B$, for $\Rd^{\rho}\B$.

We let $TA_{\alpha}$ denote $Mod(\Sigma_{\alpha})$
\begin{theorem} For any infinite ordinal $\alpha$, $TA_{\alpha}=PTA_{\alpha}.$
\end{theorem}
\begin{proof} {\bf First proof}
\begin{enumarab}
\item First for $\A\models \Sigma_{\alpha}$ and $\rho:n\to \alpha,$ $n\in \omega$ and $\rho$ one to one, define $\Rd^{\rho}\A$ as in
 \cite{HMT1} def. 2.6.1.
Then $\Rd^{\rho}\A\in TA_n$.

\item For any $n\geq 2$ and $\rho:n\to \alpha$ as above, $TA_n\subseteq\mathbf{S}\Rd^{\rho}PTA_{\alpha}$ as in \cite{HMT2} theorem 3.1.121.

\item $PTA_{\alpha}$ is closed under ultraproducts, cf \cite{HMT2}, lemma 3.1.90.

\end{enumarab}

Now we show that if $\A\models \Sigma_{\alpha}$, then $\A$ is representable.
First, for any $\rho:n\to \alpha$, $\Rd^{\rho}\A\in TA_n$. Hence it is in $PTA_n$ and so it is in $S\Rd^{\rho}PTA_{\alpha}$.
Let $I$ be the set of all finite one to one sequences with range in $\alpha$.
For $\rho\in I$, let $M_{\rho}=\{\sigma\in I:\rho\subseteq \sigma\}$.
Let $U$ be an ultrafilter of $I$ such that $M_{\rho}\in U$ for every $\rho\in I$.
Then for $\rho\in I$, there is $\B_{\rho}\in PTA_{\alpha}$ such that
$\Rd^{\rho}\A\subseteq \Rd^{\rho}\B_{\rho}$. Let $\C=\prod\B_{\rho}/U$; it is in $\mathbf{Up}PTA_{\alpha}=PTA_{\alpha}$.
Define $f:\A\to \prod\B_{\rho}$ by $f(a)_{\rho}=a$ , and finally define $g:\A\to \C$ by $g(a)=f(a)/U$.
Then $g$ is an embedding.
\end{proof}
The \textbf{second proof} follows from the next lemma,
whose proof is identical to the finite dimensional case with obvious modifications.
Here, for $\xi\in {}^\alpha\alpha^{Id},$ the operator $S_\xi$ works as $S_{\xi\upharpoonright J}$ (which can be defined as in Def. 2.12) where $J=\{i\in\alpha:\xi(i)\neq i\}$ (in case $J$ is empty, i.e., $\xi=Id_\alpha,$ $S_\xi$ is the identity operator).
\begin{lemma}\label{f}
Let $\A$ be a $TA_\alpha$ type $BAO$ and $G\subseteq{}^\alpha\alpha^{Id}$ permutable.
Let $\langle\mathcal{F}_\xi:\xi\in G\rangle$ is a system of ultrafilters of $\A$
such that for all $\xi\in G,\;i\neq j\in \alpha$ and $a\in\A$
the following condition holds:$$S_{ij}^\A(a)\in\mathcal{F}_\xi\Leftrightarrow a\in \mathcal{F}_{\xi\circ[i,j]}\quad\quad (*).$$
Then the following function $h:\A\longrightarrow\wp(G)$
is a homomorphism
$$h(a)=\{\xi\in G: a\in \mathcal{F}_\xi\}.$$
\end{lemma}

Using this lemma one proves the above theorem by the method used for the finite dimensional case replacing $S_n$ by
$S=\{s\in{}^\alpha\alpha^{Id}:s\mbox{ bijective }\}$ which is permutable.

\begin{corollary} ${\bold H}RTA_{\alpha}=\mathbf{HSP}(\{\wp(V):V\subseteq{}^\alpha\alpha^{Id}\mbox{ permutable}\})$
\end{corollary}

Theorems on complete representations also generalize verbatim with the same proofs using the technique in \ref{f}.
Let $covK$ denote the cardinal defined in the proof of theorem \ref{t}. Then we have
\begin{theorem}
\begin{enumarab}
\item Let $\A\in TA_{\alpha}$ be a countable transposition algebra $a\in A$ non-zero, and $X_i\subseteq A$, $i<covK$,
 such that $\prod X_i=0$.
Then there exists a transposition set algebra $\B$ with a weak unit and a representation
$h:\A\to \B$ such that $\bigcap_{x\in X}h(x)=\emptyset$ and $h(a)\neq 0$.
\item Let $\A$ and $(X_i:i<covK)$ be as above. Then there exists a permutable $V$ and an embedding
$f:\A\to \wp(V)$ such that $\bigcup_{x\in X_i}f(x)=\emptyset$
\end{enumarab}
\end{theorem}

The next theorem is also in essence an omitting types theorem.
First a definition. Let $\B$ be a substitution algebra.
Say that an ultrafilter $F$ in $\B$ is realized in the representation $f:\B\to \wp(V)$ if $\bigcap_{x\in F}f(x)\neq \emptyset.$

\begin{theorem} Let $\B\in TA_{\omega}$ be countable. Then there exists two representations of $\B$ such that if $F$
is an ultrafilter realized in both,
then $F$ is principal.
\end{theorem}
\begin{proof} We construct two distinct representations of $\B$ such that if $F$ is an ultrafilter in $B$ that is realized in both representations, then $F$ is
necessarily principal, that is $\prod F$ is an atom generating $F$.
We construct two ultrafilter $T$ and $S$ of $\B$ such that (*)
$\forall \tau_1, \tau_2\in {}^{\omega}\omega^{(Id)}( G_1=\{a\in \B: {\sf s}_{\tau_1}a\in T\},
G_2=\{a\in \B: s_{\tau_1}a\in S\})\\\implies G_1\neq G_2 \text { or $G_1$ is principal.}$
Note that $G_1$ and $G_2$ are indeed ultrafilters.
We construct $S$ and $T$ as a union of a chain. We carry out various tasks as we build the chains.
The tasks are as in (*), as well as

(**) for all $a\in A$ either $a\in T$ or $-a\in T$, and same for $S$.

We let $S_0=T_0=\{1\}$.
There are countably many tasks. Metaphorically we hire countably many experts and give them one task each.
We partition $\omega$ into infinitely many sets and we assign one of these tasks to each expert.
When $T_{i-1}$ and $S_{i-1}$ have been chosen and $i$ is in the set assigned to some expert $E$, then $E$ will construct
$T_i$ and $S_i$.
For consider the expert who handles task (***). Let $X$ be her subset of $\omega$. Let her list as $(a_i: i\in X)$ all elements of $X$.
When $T_{i-1}$ has been chosen with $i\in X$, she should consider whether $T_{i-1}\cup \{a_i\}$ is consistent. If it is she puts
$T_i=T_{i-1}\cup \{a_i\}$. If not she puts $T_i=T_{i-1}\cup \{-a_i\}$. Same for $S_i$.
Now finally consider the tasks in (*). Suppose that $X$ contains $i$ , and $S_{i-1}$ and $T_{i-1}$ have been chosen.
Let $e=\bigwedge S_{i-1}$ and $f=\bigwedge T_{i-1}$. We have two cases.
If $e$ is an atom in $B$ then the ultrafilter $F$ containing $e$ is principal so our expert can put $S_i=S_{i-1}$ and $T_i=T_{i-1}$.
If not, then let $F_1$ , $F_2$ be distinct ultrafilters containing $e$. Let $G$ be an ultrafilter containing $f$. Say $F_1$ is different from $G$.
Let $\theta$ be in $F_1-G$. Then put $S_i=S_{-1}\cup \{\theta\}$ and $T_i=T_{i-1}\cup \{-\theta\}.$
It is not hard to check that the canonical models, defined the usual way,  corresponding to $S$ and $T$ are as required.
\end{proof}
\subsection*{Remark}

The above technique is important in omitting types theorems, since it can, in certain contexts, allow omitting $\kappa< {} ^{\omega}2$ types,
given that they are maximal
(i.e ultrafilters). The idea is to construct ${}^{\omega}2$ pairwise non-isomorphic models,
such that given any ultrafilter $F$ realized in two of them
is necessarily principal.
One then defines for $i<{}^{\omega}2$,  $K_i=\{\M: \M \text { omits }F_i\}$, and for limits $K_{\mu}=\bigcap_{i<\mu}K_i$.
Since  $\kappa < {}^{\omega}2$, there will be a model $\M\in \bigcap_{i< ^{\omega}2}K_i$, and $\M$ will omit all given types.

\begin{theorem} Atomic transposition algebras of infinite dimensions  are completely representable on weak sets
\end{theorem}
\begin{proof} Like the finite dimensional case.
\end{proof}

Finally we show that:

\begin{theorem} For infinite ordinals $\alpha$, $RQA_{\alpha}$ is not a variety.
\end{theorem}
\begin{proof} Assume to the contrary that $RQA_{\alpha}$ is a variety and that $RQA_{\alpha}={\bf Mod}\Sigma_{\alpha}$ for some countable schema
$\Sigma_{\alpha}.$ Fix $n\geq 2.$ We show that for any set $U$ and any ideal $I$ of $\A=\wp(^nU)$, we have $\A/I\in RQA_n$,
which is not possible since we know that there are relativized set algebras to permutable sets that are not in $RQA_n$.
Define $f:\A\to \wp(^{\alpha}U)$ by $f(X)=\{s\in {}^{\alpha}U: f\upharpoonright n\in X\}$. Then $f$ is an embedding of $\A$ into
$\Rd_n(\wp({}^nU))$, so that we can assume that
$\A\subseteq \Rd_n\B$, for some $\B\in RQA_{\alpha}.$ Let $I$ be an ideal of $\A$, and let $J=\Ig^{\B}I$. Then we claim that
$J\cap \A=I$. One inclusion is trivial; we need to show $J\cap \A\subseteq I$. Let $y\in A\cap J$. Then $y\in \Ig\{I\}$ and so,
there is a term $\tau$, and $x_1,\ldots x_n\in I$ such that $y\leq \tau(x_1,\dots x_n)$. But $\tau(x_1,\ldots x_{n-1})\in I$ and $y\in A$,
hence $y\in I$, since ideals are closed downwards.
It follows that $\A/I$ embeds into $\Rd_n(\B/J)$
via $x/I\mapsto x/J$. The map is well defined since $I\subseteq J$, and it is one to one, because if $x,y\in A$, such that $x\delta y\in J$,
then $x\delta y\in I$. We have $\B/J\models \Sigma_{\alpha}$.

For $\beta$ an ordinal,
let $K_{\beta}$ denote the class of all full set algebras of dimension $\beta$. Then $K_n=S\Rd_nK_{\alpha}$. One inclusion is obvious.
To prove the other inclusion, it suffices to show that that if $\A\subseteq \Rd_n(\wp({}^{\alpha}U))$,
then $\A$ is embeddable in $\wp(^nW)$, for some set $W$.
Simply take $W=U$ and define $g:\A\to \wp(^nU)$ by $g(X)=\{f\upharpoonright n: f\in X\}$.

Now let $\B'=\B/I$, then $\B'\in {\bf SP}K_{\alpha}$, so $\Rd_n\B'\in \Rd_n{\bf SP}K_{\alpha}={\bf SP}\Rd_nK_{\alpha}\subseteq {\bf SP}K_n$.
Hence $\A/I\in RTA_n$. But this cannot happen for all $\A\in K_n$ and we are done.
\end{proof}

\section{Substitution Algebras with Transpositions}

In this section we study the common expansion of S\'agi's algebras and ours when we have all replacements and substitutions.
We provide a finite axiomatization of the variety of the generated by the set algebras.
Unlike the other two cases, we show that the class of subdirect products of set algebras is a variety, for finite as well as of for infinite dimensions

\begin{definition}[Substitution Set Algebras with Transpositions]
Let $U$ be a set. \emph{The full substitution set algebra with transpositions of dimension} $\alpha$ \emph{with base} $U$ is the algebra $$\langle\mathcal{P}({}^\alpha U); \cap,-,S^i_j,S_{ij}\rangle_{i\neq j\in\alpha}.$$
Where $S^i_j$'s are as in Definition 2.10 in \cite{sagiphd} and $S_{ij}$'s are the unary operations defined before. The class of Substitution Set Algebras with Transpositions of dimension $\alpha$ is defined as follows:
$$SetSA_\alpha=\mathbf{S}\{\A:\A\text{ is a full substitution set algebra with transpositions} $$$$\text{of dimension }\alpha  \text{ with base }U,\text{ for some set }U\}.$$
\end{definition}

\begin{definition}[Representable Substitution Set Algebras with Transpositions]
The class of Substitution Set Algebras with Transpositions of dimension $\alpha$ is defined to be $$RSA_\alpha=\mathbf{SP}SetSA_\alpha.$$
\end{definition}

We  adapt the definition of ``Permutable Set", in the new context, as expected:
\begin{definition}[Dipermutable Set]
Let $U$ be a given set, and let $D\subseteq{}^\alpha U.$ We say that $D$ is dipermutable iff it satisfies the following condition:
$$(\forall i\neq j\in\alpha)(\forall s\in{}^\alpha U)(s\in D\Longrightarrow s\circ [i/j]\mbox{ and }s\circ [i,j]\in D),$$

\end{definition}

\begin{definition}[Dipermutable Algebras]
The class of \emph{Dipermutable Set Algebras} (of dimension $n<\omega$) is defined to be $$DPSA_n=\mathbf{SP}\{\langle\mathcal{P}(D); \cap,-,S^i_j,S_{ij}\rangle_{i\neq j\in n}: U\text{ \emph{is a set}}, D\subset{}^\alpha U\text{\emph{permutable}}\}$$
Here $S_j^i(X)=\{q\in D:q\circ [i/j]\in X\}$ and $S_{ij}(X)=\{q\in D:q\circ [i,j]\in X\}$, and $-$ is complement w.r.t. $D$.\\
If $D$ is a dipermutable set then the algebra $\wp(D)$ is defined to be $$\wp(D)=\langle\mathcal{P}(D);\cap,-,S^i_j,S_{ij}\rangle_{i\neq j\in n}$$ So $\wp(D)\in DPSA_n.$
\end{definition}
We later define Dipermutable Set Algebras of infinite dimension $\alpha$. Analogous to our earlier investigations, we have:

\begin{theorem}
Let $U$ be a set and suppose $G\subseteq{}^n U$ is dipermutable. Let $\A=\langle\mathcal{P}({}^n U);\cap,-,S^i_j,S_{ij}\rangle_{i\neq j\in n}$ and let $\mathcal{B}=\langle\mathcal{P}(G);\cap,-,S^i_j,S_{ij}\rangle_{i\neq j\in n}$. Then the following function $h$ is a homomorphism. $$h:\A\longrightarrow\mathcal{B},\quad h(x)=x\cap G.$$
\end{theorem}

\begin{definition}[Small algebras]
For any natural number $k\leq n$ the algebra $\A_{nk}$ is defined to be $$\A_{nk}=\langle\mathcal{P}({}^nk);\cap,-,S^i_j,S_{ij}\rangle_{i\neq j\in n}.$$ So $\A_{nk}\in SetSA_n$.
\end{definition}

\begin{theorem}
$RSA_n=\mathbf{SP}\{\A_{nk}:k\leq n\}.$
\end{theorem}
\begin{proof}
Exactly like before.
\end{proof}

\subsection{Axiomatizing $RSA_n$.}
In this section we show that $RSA_n$ is a variety by giving a set $\Sigma'_n$ of equations such that $\mathbf{Mod}(\Sigma'_n)=RSA_n$.
\begin{definition}[The Axiomatization]
For all natural $n>1$, let $\Sigma'_n$ be the set of equations $\Sigma_n$ defined before together with the following: For distinct $i,j,k,l$
\begin{enumerate}
\item$s_{kl}s^j_is_{kl}x=s^j_ix$
\item$s_{jk}s^j_is_{jk}x=s^k_ix$
\item$s_{ki}s^j_is_{ki}x=s^j_kx$
\item$s_{ij}s^j_is_{ij}x=s^i_jx$
\item$s^j_is^k_lx=s^k_ls^j_ix$
\item$s^j_is^k_ix=s^k_is^j_ix=s^j_is^k_jx$
\item$s^j_is^i_kx=s^j_ks_{ij}x$
\item$s^j_is^j_kx=s^j_kx$
\item$s^j_is^j_ix=s^j_ix$
\item$s^j_is^i_jx=s^j_ix$
\item$s^j_is_{ij}x=s_{ij}x$
\end{enumerate}
\end{definition}
We let $SA_n$ denote $Mod(\Sigma'_n).$
\begin{definition}
Let $R(U)=\{s_{ij}:i\neq j\in U\}\cup\{s^i_j:i\neq j\in U\}$ and let $\hat{}:R(U)^*\longrightarrow {}^UU$ defined inductively as follows: it maps the empty string to $Id_U$ and for any string $t$, $$(s_{ij}t)^{\hat{}}=[i,j]\circ t^{\hat{}}\;\;and\; (s^i_jt)^{\hat{}}=[i/j]\circ t^{\hat{}}.$$
\end{definition}

\begin{theorem}
For all $n\in\omega$ the set of (all instances of the) axiom-schemas 4,5,6 of Def.3.8 and 1 to 11 of Def.4.8
is a presentation of the semigroup ${}^nn$ via generators $R(n)$.
That is, for all $t_1,t_2\in R(n)^*$ we have $$4,5,6\mbox{ of Def.3.8 and 1 to 11 of Def.4.8 }\vdash t_1=t_2\text{  iff  }t_1^{\hat{}}=t_2^{\hat{}}.$$
Here $\vdash$ denotes derivability using Birkhoff's calculus for equational logic.
\end{theorem}
\begin{proof}
This is clear because the mentioned schemas correspond exactly to the set of relations governing the generators
of ${}^nn$ (see \cite{semigroup}).
\end{proof}

\begin{definition}For every $\xi\in {}^nn$ we associate a sequence $s_\xi\in R(U)^*$ such that $s_\xi^{\hat{}}=\xi.$ Such an $s_\xi$ exists, since $R(n)$ generates ${}^nn.$ This is a combination of Def.2.12 here and Def.4.13 in \cite{sagiphd}.
\end{definition}

Like before, we have
\begin{lemma}\label{lemma}
Let $\A$ be an $RSA_n$ type $BAO$. Suppose $G\subseteq {}^nn$ is a dipermutable set, and $\langle\mathcal{F}_\xi:\xi\in G\rangle$ is a system of ultrafilters of $\A$ such that for all $\xi\in G,\;i\neq j\in n$ and $a\in\A$ the following conditions hold:$${S_{ij}}^\A(a)\in\mathcal{F}_\xi\Leftrightarrow a\in \mathcal{F}_{\xi\circ[i,j]}\quad\quad (*),\text{ and}$$ $${S^i_j}^\A(a)\in\mathcal{F}_\xi\Leftrightarrow a\in \mathcal{F}_{\xi\circ[i/j]}\quad\quad (**)$$  Then the following function $h:\A\longrightarrow\wp(G)$ is a homomorphism$$h(a)=\{\xi\in G:a\in \mathcal{F}_\xi\}.$$
\end{lemma}

Now, we show, unlike replacement algebras,  that $RSA_n$ is a variety.
\begin{theorem}\label{variety} For any finite $n\geq 2$,
$RSA_n=SA_n$
\end{theorem}
\begin{proof}
Clearly, $RSA_n\subseteq SA_n$ because $SetSA_n\models\Sigma'_n$ (checking it is a routine computation). Conversely, $RSA_n\supseteq SA_n$. To see this, let $\A\in SA_n$ be arbitrary. We may suppose that $\A$ has at least two elements, otherwise, it is easy to represent $\A$. For every $0^\A\neq a\in A$ we will construct a homomorphism $h_a$ on $\A_{nn}$ such that $h_a(a)\neq 0^{\A_{nn}}$.

To do this, let $0^\A\neq a\in A$ be an arbitrary element. Let $\mathcal{F}$ be an ultrafilter over $\A$ containing $a$, and for every $\xi\in {}^nn$ let $\mathcal{F}_\xi=\{z\in A: S^\A_\xi(z)\in\mathcal{F}\}$ (which is an ultrafilter). Then, $h:\A\longrightarrow\A_{nn}$ defined by $h(z)=\{\xi\in{}^nn:z\in\mathcal{F}_{\xi}\}$ is a homomorphism by \ref{lemma} as $(*),$ $(**)$ hold.
\end{proof}
Also here, by slight modifications to arguments in the previous section,
we have $$SA_n\subseteq DPSA_n\subseteq\mathbf{HSP}RSA_n$$
which is equivalent to that $$SA_n= DPSA_n=RSA_n.$$

\subsection{Complete representations}

For $SA_n$, the problem of complete representations is more delicate,
since the substitutions corresponding to replacements are not necessarily complete endomorphisms, or at least we could not prove that
they are. However,  we could obtain several results in this context on complete representations.

\begin{definition} Let $\A\in SA_n$ and $b\in A$, then $\Rl_{b}\A=\{x\in \A: x\leq b\}$, with operations relativized to $b$.
\end{definition}

\begin{theorem}\label{atomic} Let $\A\in SA_n$ is atomic  and assume that $\sum_{x\in X} s_{\tau}x=b$ for all $\tau\in {}^nn$.
Then $\Rl_{b}\A$ is completely representable.
In particular,  if $\sum_{x\in x}s_{\tau}x=1$, then $\A$ is completely representable.
\end{theorem}
\begin{proof} Clearly $\B=\Rl_{b}\A$ is atomic and for $\tau\in {}^nn$, $\sum_{x\in X}s_{\tau}(x.b)=b$.
\end{proof}
For an algebra $\A$, $\A^+$ denotes its canonical extension.

\begin{theorem}\label{canonical} Let $\A\in SA_n$. Then $\A^+$ is completely representable. In fact, any representation of $\A$ is complete.
\end{theorem}
\begin{proof} Let $X=\At\A^+$. Then $\sum X=1$, and so there exists a finite $X'\subset X$
such that  $\sum X'=\sum X=1$, and so $\sum { s}_{\tau}X={s}_{\tau}\sum X=1$ for every $\tau\in {}^nn$.
\end{proof}
\begin{lemma} For $\A\in SA_n$, the following two conditions are equivalent:
\begin{enumarab}
\item There exists a dipermutable set $V$, and a complete representation $f:\A\to \wp(V)$.
\item For all non zero $a\in A$, there exists a homomorphism $f:\A\to \wp(^nn)$ such that $f(a)\neq 0$, such that $\bigcup_{x\in \At\A} f(x)={}^nn$.
\end{enumarab}
\end{lemma}
\begin{demo}{Proof} We have already proved (2) implies (1). Conversely, let there be given a complete representation $g:\A\to \wp(V)$.
Then $\wp(V)\subseteq \prod_{i\in I} \A_i$ for some set $I$, where $\A_i=\wp{}(^nn)$. Assume that $a$ is non-zero,
then $g(a)$ is non-zero, hence $g(a)_i$ is non-zero
for some $i$. Let $\pi_j$ be the $j$th projection $\pi_j:\prod \A_i\to \A_i$, $\pi_j[(a_i: i\in I)]=a_j$.
Define $f:\A\to \A_i$ by $f(x)=(\pi_i\circ g(x)).$
Then clearly $f$ is as required.
\end{demo}

The following theorem is a converse to \ref{atomic}
\begin{theorem}\label{converse} Assume that $\A$ is a substitution algebra that is completely representable.
Let $f:\A\to \wp(V)$ be a complete representation. Then $\sum s_{\tau}\At\A=1$ for every $\tau\in {}^nn$.
\end{theorem}
\begin{proof} Let $a\in A$, and $f:\A\to \wp(^nn)$, such that
$f(a)\neq 0$, and $\bigcup_{x\in \At\A} f(x)={}^nn$.
Let $F=\{a\in A: Id\in f(a)\}$. Then $F$ is an ultrafilter.
For $a\in A$, let $rep_F(a)=\{\tau\in {}^nn: s_{\tau}a\in F\}$.
Then $rep_F:\A\to \wp(V)$ and $rep_F=f$. Indeed for $b\in \A$, and $\tau\in {}^nn$, we have $\tau\in rep_F(b)$ iff
$s_{\tau}b\in F$ iff $Id\in f(s_{\tau}b)$ iff $\tau \in f(b).$ Assume, seeking a contradiction, that there exist $\tau\in {}^nn$ and $y\in A$,
$y<1$, such that $s_{\tau}x\leq y$ for all $x\in X$.
Then for all $x\in X$, $\tau\notin rep_{F}(x)$, for if $x\in rep_F(x)$, then there would be an $x\in X$, such that $\tau\in rep_F(x)$, so that
$s_{\tau}x=1$, and this is not possible. This means that $\bigcup_{x\in X} rep_F(x)\neq {}^nn$, and this contradicts complete representability.
\end{proof}

There is another kind of completion for a $BAO$  called its  minimal completion. For an algebra $\A$, its canonical extension and minimal
completion coincide if and only if
$\A$ is finite. if $\A$ is atomic, then its minimal completion is easy to construct; it is just the complex algebra of its atom structure.
If $\A$ is an algebra and $\B$ is its minimal completion, then for any $X\subseteq \A$, we have $\sum^{\A}X=\sum^{\B}X$, whenever the former exists.
\begin{theorem}\label{completion} If $\A$ is completely representable, then so is its minimal completion
\end{theorem}
\begin{proof} Let $\A$ be given and let $\B$ denote its minimal completion. Then $\A$ is atomic, and since
$\A$ is completely representable, then by theorem \ref{converse}
$\sum s_{\tau}^{\A}\At\A=1$ for every $\tau\in V$. But, suprema are preserved in the completion, hence
we have, $\sum s_{\tau}^{\B}\At\B=1$ for every such $\tau$ and
we are done by \ref{atomic}.
\end{proof}

\subsection{The Infinite Dimensional Case}

Also here, for $SA$'s, we can lift our results to infinite dimensions.
Surprisingly, while we could not capture the extension by diagonal elements in the finite dimensional case, it turns out that here we can,
when we enough spare dimensions; in fact we have infinitely many.

\begin{definition}
We let $DPSA_{\alpha}$ be the variety generated by
$$\wp(V)=\langle\mathcal{P}(V),\cap,-,S^i_j,S_{ij}\rangle_{i,j\in\alpha}, \ \ V\subseteq{}^\alpha\alpha^{Id}$$ is dipermutable.

\end{definition}
In the next theorem we show that weak spaces can be squared using a non-trivial ultraproduct construction.
Let $\Sigma_{\alpha}$ be the set of finite schemas obtained from
from the $\Sigma_n$ but now allowing indices from $\alpha$.  We know that
if $\A\subseteq \wp(^{\alpha}U)$ and $a\in A$ is non zero,
then there exists a homomorphism $f:\A\to \wp(V)$ for some permutable $V$ such that $f(a)\neq 0$.
We now prove a converse of this result.
But first a definition and a result on the number of non-isomorphic models.
\begin{definition}
Let $\A$ and $\B$ be set algebras with bases $U$ and $W$ respectively. Then $\A$ and $\B$
are base isomorphic if there exists a bijection
$f:U\to W$ such that $\bar{f}:\A\to \B$ defined by ${\bar f}(X)=\{y\in {}^{\alpha}W: f^{-1}\circ y\in x\}$ is an isomorphism from $\A$ to $\B$
\end{definition}
\begin{definition} An algebra $\A$ is hereditary atomic, if each of its subalgebras is atomic.
\end{definition}
Finite Boolean algebras are hereditary atomic of course,
but there are infinite hereditary atomic Boolean algebras. What characterizes such algebras is that the base of their Stone space, that is the set of all
ultrafilters is countable.
\begin{theorem} Let $\A\in SA_{\omega}$ be countable and simple
Then the number of non base isomorphic representations of $\A$ is either $\leq \omega$ or $^{\omega}2$.
\end{theorem}
\begin{proof} If $\A$ is hereditary atomic, then the number of models $\leq$ the number of ultrafilters.
Else, $\A$ is non-atomic, then it has $^{\omega}2$ ultrafilters. For an ultrafilter $F$, let $h_F(a)=\{\tau \in V: s_{\tau}a\in F\}$.
Then $h_F:\A\to \wp(V)$. We have $h_F(\A)$ is base isomorphic to $h_G(\A)$ iff there exists a finite bijection $\sigma\in V$ such that
$s_{\sigma}F=G$. Define the equivalence relation $\sim $ on the set of ultrafilters by $F\sim G$, if there exists a finite permutation $\sigma$
such that $F=s_{\sigma}G$. Then any equivalence class is countable, and so we have $^{\omega}2$ many orbits, which correspond to
the non base isomorphic representations of $\A$.
\end{proof}
We shall prove that weak set algebras are strongly isomorphic to set algebras in the sense of the following definition.

\begin{definition}
Let $\A$ and $\B$ be set algebras with units $V_0$ and $V_0$ and bases $U_0$ and $U_1,$ respectively,
and let $F$ be an isomorphism from $\B$ to $\A$.
Then $F$ is a strong ext-isomorphism if $F=(X\cap V_0: X\in B)$. In this case $F^{-1}$ is called a strong subisomorphism. An isomorphism
$F$ from $\A$ to $\B$ is a strong ext base isomorphism if $F=g\circ h$
for some base isomorphism and some strong ext isomorphism $g$. In this case $F^{-1}$ is called a strong sub base isomorphism.

\end{definition}

\begin{theorem} If $\B$ is a subalgebra of $ \wp(^{\alpha}\alpha^{(Id)})$ then there exists a set algebra $\C$ with unit $^{\alpha}U$
such that $\B\cong \C$. Furthermore, the isomorphism is a strong sub base isomorphism.
\end{theorem}

\begin{proof}We square the unit using ultraproducts.
We prove the theorem for $\alpha=\omega$. Let $F$ be a non-principal ultrafilter
over $\omega$. Then there exists a function
$h: \omega\to \{\Gamma\subseteq_{\omega} \omega\}$
such that $\{i\in \omega: \kappa\in h(i)\}\in F$ for all $\kappa<\omega$.
Let $M={}^{\omega}U/F$.  $M$ will be the base of our desired algebra, that is  $\C$ will
have unit $^{\omega}M.$
Define $\epsilon: U\to {}^{\omega}U/F$ by
$$\epsilon(u)=\langle u: i\in \omega\rangle/F.$$
Then it is clear that $\epsilon$ is one to one.
For $Y\subseteq {}^{\omega}U$,
let $$\bar{\epsilon}(Y)=\{y\in {}^{\omega}(^{\omega}U/F): \epsilon^{-1}\circ y\in Y\}.$$
By an $(F, (U:i\in \omega), \omega)$ choice function we mean a function
$c$ mapping $\omega\times {}^{\omega}U/F$
into $^{\omega}U$ such that for all $\kappa<\omega$
and all $y\in {}^{\omega}U/F$, we have $c(k,y)\in y.$
Let $c$ be an $(F, (U:i\in \omega), \omega)$
choice function satisfying the following condition:
For all $\kappa, i<\omega$ for all $y\in X$, if
$\kappa\notin h(i)$ then $c(\kappa,y)_i=\kappa$,
if $\kappa\in h(i)$ and $y=\epsilon u$ with  $u\in U$ then $c(\kappa,y)_i=u$.
Let $\delta: \B\to {}^{\omega}\B/F$ be the following monomorphism
$$\delta(b)=\langle b: i\in \omega\rangle/F.$$
Let $t$ be the unique homomorphism
mapping
$^{\omega}\B/F$ into $\wp{}^{\omega}(^{\omega}U/F)$
such that  for any $a\in {}^{\omega}B$
$$t(a/F)=\{q\in {}^{\omega}(^{\omega}U/F): \{i\in \omega: (c^+q)_i\in a_i\}\in F\}.$$
Here $(c^+q)_i=\langle c(\kappa,q_\kappa)_i: k<\omega\rangle.$
Let $g=t\circ \delta$. Then for $a\in B$,
$$g(a)=\{q\in {}^{\omega}(^{\omega}U/F): \{i\in \omega: (c^+q)_i\in a\}\in F\}.$$
Let $\C=g(\B)$. Then $g:\B\to \C$.
We show that $g$ is an isomorphism
onto a set algebra. First it is clear that $g$ is a monomorphism into an algebra with unit $g(V)$.
Recall that $M={}^{\omega}U/F$. Evidently $g(V)\subseteq {}^{\omega}M$.
We show the other inclusion. Let $q\in {}^{\omega}M$. It suffices to show that
$(c^+q)_i\in V$ for all $i\in\omega$. So, let $i\in \omega$. Note that
$(c^+q)_i\in {}^{\omega}U$. If $\kappa\notin h(i)$ then we have
$$(c^+q)_i\kappa=c(\kappa, q\kappa)_i=\kappa.$$
Since $h(i)$ is finite the conclusion follows.
We now prove that for $a\in B$
$$(*) \ \ \ g(a)\cap \bar{\epsilon}V=\{\epsilon\circ s: s\in a\}.$$
Let $\tau\in V$. Then there is a finite $\Gamma\subseteq \omega$ such that
$$\tau\upharpoonright (\omega\smallsetminus \Gamma)=
p\upharpoonright (\omega\smallsetminus \Gamma).$$
Let $Z=\{i\in \omega: \Gamma\subseteq hi\}$. By the choice of $h$ we have $Z\in F$.
Let $\kappa<\omega$ and $i\in Z$.
We show that $c(\kappa,\epsilon\tau \kappa)_i=\tau \kappa$.
If
$\kappa\in \Gamma,$ then $\kappa\in h(i)$ and so
$c(\kappa,\epsilon \tau \kappa)_i=\tau \kappa$. If $\kappa\notin \Gamma,$
then $\tau \kappa=\kappa$
and $c(\kappa,\epsilon \tau \kappa)_i=\tau\kappa.$
We now prove $(*)$. Let us suppose that $q\in g(a)\cap {\bar{\epsilon}}V$.
Since $q\in \bar{\epsilon}V$ there is an $s\in V$
such that $q=\epsilon\circ s$.
Choose $Z\in F$
such that $$c(\kappa, \epsilon(s\kappa))\supseteq\langle s\kappa: i\in Z\rangle$$
for all $\kappa<\omega$. This is possible by the above.
Let $H=\{i\in \omega: (c^+q)_i\in a\}$.
Then $H\in F$. Since $H\cap Z$ is in $F$
we can choose $i\in H\cap Z$.
Then we have
$$s=\langle s\kappa: \kappa<\omega\rangle=
\langle c(\kappa, \epsilon(s\kappa))_i:\kappa<\omega\rangle=
\langle c(\kappa,q\kappa)_i:\kappa<\omega\rangle=(c^+q)_i\in a.$$
Thus $q\in \epsilon \circ s$. Now suppose that $q=\epsilon\circ s$ with $s\in a$.
Since $a\subseteq V$ we have $q\in \epsilon V$.
Again let $Z\in F$ such that for all $\kappa<\omega$
$$c(\kappa, \epsilon
s \kappa)\supseteq \langle s\kappa: i\in Z\rangle.$$
Then $(c^+q)_i=s\in a$ for all $i\in Z.$ So $q\in g(a).$
Note that $\bar{\epsilon}V\subseteq {}^{\omega}(^{\omega}U/F)$.
Let $rl_{\epsilon(V)}^{\C}$ be the function with domain $\C$
(onto $\bar{\epsilon}(\B))$
such that
$$rl_{\epsilon(V)}^{\C}Y=Y\cap \bar{\epsilon}V.$$
Then we have proved that
$$\bar{\epsilon}=rl_{\bar{\epsilon V}}^{\C}\circ g.$$
It follows that $g$ is a strong sub-base-isomorphism of $\B$ onto $\C$.
\end{proof}
Like the finite dimensional case, we get:
\begin{corollary}\label{v2} $\mathbf{SP}\{ \wp(^{\alpha}U): \text {$U$ a set }\}$ is a variety.
\end{corollary}
\begin{proof} Let $\A\in SA_\alpha$. Then for $a\neq 0$ there exists a weak set algebra $\B$ and $f:\A\to \B$ such that $f(a)\neq 0$.
By the previous theorem there is a set algebra $\C$ such that $\B\cong \C$, via $g$ say. Then $g\circ f(a)\neq 0$, and we are done.
\end{proof}

\subsection{Adding Diagonals}

\begin{definition}Let $\Sigma'^d_{\alpha}$ be the axiomatization obtained by adding to $\Sigma'_{\alpha}$ the following equations for al $i,j<\alpha$.
\begin{enumerate}
\item $d_{ii}=1$
\item $d_{i,j}=d_{j,i}$
\item $d_{i,k}.d_{k,j}\leq d_{i,j}$
\item $s_{\tau}d_{i,j}=d_{\tau(i), \tau(i)}$, $\tau\in \{[i,j], [i|j]\}$.
\end{enumerate}
\end{definition}
\begin{theorem}\label{infinite}
Every substitution algebra with transpositions and diagonals of infinite dimension is representable.
\end{theorem}
\begin{proof}
Let $\A\in\mathbf{Mod}(\Sigma'^d_{\alpha})$
and let $0^\A\neq a\in A$. We construct a homomorphism $h:\A\longrightarrow\wp (^{\alpha}\alpha^{(Id)})$.
such that $h(a)\neq 0$.
Like before, choose an ultrafilter $\mathcal{F}\subset A$ containing $a$. Let $h:\A\longrightarrow \wp(^{\alpha}\alpha^{(Id)})$
be the following function $h(z)=\{\xi\in ^{\alpha}\alpha^{(Id)}:S_{\xi}^\A(z)\in\mathcal{F}\}.$
Now, if $\mathcal{F}_\xi=\{z\in A: S_\xi(z)\in\mathcal{F}\},$ then $\langle\mathcal{F}_\xi:\xi\in ^{\alpha}\alpha^{(Id)}\rangle$ is a system of ultrafilters
of $\A$ satisfying $(*)$.  The function $h$ respects substitutions but it may not respect the diagonal elements.
To ensure that it does we factor out $\alpha$, the base of the set algebra, by a congruence relation.
Define the following equivalence relation $\sim$ on $\alpha$, $i\sim j$ iff $d_{ij}\in F$. Using the axioms for diagonals $\sim$
is an equivalence relation.
Let $V={}^{\alpha}\alpha{(Id}),$ and  $M=V/\sim$. For $h\in V$ we write
$h=\bar{\tau}$, if $h(i)=\tau(i)/\sim$ for all $i\in n$. Of course $\tau$ may not be unique.
Now define $f(z)=\{\bar{\xi}\in M: S_{\xi}^{\A}(z)\in \mathcal{F}\}$. We first check that $f$ is well defined.
We use extensively the property $(s_{\tau}\circ s_{\sigma})x=s_{\tau\circ \sigma}x$ for all
$\tau,\sigma\in {}^{\alpha}\alpha^{(Id)}$, a property that can be inferred form our axiomatization.
We show that $f$ is well defined, by induction on the cardinality of
$$J=\{i\in \mu: \sigma (i)\neq \tau (i)\}.$$
Of course $J$ is finite. If $J$ is empty, the result is obvious.
Otherwise assume that $k\in J$. We introduce a piece of notation.
For $\eta\in V$ and $k,l<\alpha$, write
$\eta(k\mapsto l)$ for the $\eta'\in V$ that is the same as $\eta$ except
that $\eta'(k)=l.$
Now take any
$$\lambda\in \{\eta\in \alpha: \sigma^{-1}\{\eta\}= \tau^{-1}\{\eta\}=\{\eta\}\}$$
We have  $${ s}_{\sigma}x={ s}_{\sigma k}^{\lambda}{ s}_{\sigma (k\mapsto \lambda)}x.$$
Also we have (b)
$${s}_{\tau k}^{\lambda}({ d}_{\lambda, \sigma k}. {\sf s}_{\sigma} x)
={ d}_{\tau k, \sigma k} { s}_{\sigma} x,$$
and (c)
$${ s}_{\tau k}^{\lambda}({ d}_{\lambda, \sigma k}.{\sf s}_{\sigma(k\mapsto \lambda)}x)$$
$$= { d}_{\tau k,  \sigma k}.{ s}_{\sigma(k\mapsto \tau k)}x.$$

and (d)

$${ d}_{\lambda, \sigma k}.{ s}_{\sigma k}^{\lambda}{ s}_{{\sigma}(k\mapsto \lambda)}x=
{ d}_{\lambda, \sigma k}.{ s}_{{\sigma}(k\mapsto \lambda)}x$$

Then by (b), (a), (d) and (c), we get,

$${ d}_{\tau k, \sigma k}.{ s}_{\sigma} x=
{ s}_{\tau k}^{\lambda}({ d}_{\lambda,\sigma k}.{ s}_{\sigma}x)$$
$$={ s}_{\tau k}^{\lambda}({ d}_{\lambda, \sigma k}.{ s}_{\sigma k}^{\lambda}
{ s}_{{\sigma}(k\mapsto \lambda)}x)$$
$$={s}_{\tau k}^{\lambda}({ d}_{\lambda, \sigma k}.{s}_{{\sigma}(k\mapsto \lambda)}x)$$
$$= { d}_{\tau k,  \sigma k}.{ s}_{\sigma(k\mapsto \tau k)}x.$$
The conclusion follows from the induction hypothesis.
Finally, clearly $f$ respects diagonal elements.

\end{proof}

\begin{question} Are the atomic algebras completely representable?
\end{question}

\section{Atomicity of free algebras}

\subsection{General results on free algebras}

In cylindric algebra theory, whether the free algebras are atomic or not is an important topic. In fact, N\'emeti proves
that for $n\geq 3$ the free algebras of dimension $n$ on a finite set of generators are not atomic,
and this is closely related to Godels incompleteness theorems for the finite $n$-variable fragments of first order logic.
We first start by proving slightly new results concerning free algebra of class of $BAO$'s.

\begin{definition}
Let $K$ be variety  of $BAO$'s. Let $\L$ be the corresponding multimodal logic.
We say that $\L$ has the Godel's incompleteness property if there exists
a formula $\phi$ that cannot be extended to a recursive complete theory.
Such formula is called incompletable.
\end{definition}
Let $\L$ be a general modal logic, and let $\Fm_{\equiv}$ be the Tarski-Lindenbaum formula algebra on
finitely many generators.
\begin{theorem}(Essentially Nemeti's) If $\L$ has $G.I$, then the algebra $\Fm_{\equiv}$
is not atomic.
\end{theorem}
\begin{proof}
Assume that $\L$ has $G.I$. Let $\phi$ be an incompletable
formula. We show that there is no atom in the Boolean algebra $\Fm$
below $\phi/\equiv.$
Note that because $\phi$ is consistent, it follows that $\phi/\equiv$ is non-zero.
Now, assume to
the contrary that there is such an atom $\tau/\equiv$ for some
formula $\tau.$
This means that .
that $(\tau\land \bar{\phi})/\equiv=\tau/\equiv$.
Then it follows that
$\vdash (\tau\land \phi)\implies \phi$, i.e.
$\vdash\tau\implies \phi$.
Let
$T=\{\tau,\phi\}$
and let
$Consq(T)=\{\psi\in Fm: T\vdash \psi\}.$
$Consq(T)$ is short for the consequences of $T$.
We show that $T$ is complete and that $Consq(T)$ is
decidable.   Let $\psi$ be an arbitrary formula in $\Fm.$
Then either $\tau/\equiv\leq \psi/\equiv$ or $\tau/\equiv\leq \neg \psi/\equiv$
because $\tau/\equiv$ is an
atom. Thus $T\vdash\psi$ or $T\vdash \neg \psi.$
Here it is the {\it exclusive or} i.e. the two cases cannot occur together.
Clearly $ConsqT$ is recursively enumerable.  By completeness of $T$ we have
$\Fm_{\equiv}\smallsetminus Consq(T)=\{\neg \psi: \psi\in Consq(T)\},$
hence the complement of $ConsqT$ is recursively enumerable as well, hence $T$
is decidable.  Here we are using the trivial fact that $\Fm$ is decidable.
This contradiction proves that $\Fm_{\equiv}$ is not atomic.
\end{proof}
\begin{definition} An element $a\in A$ is closed, if $f_i(a)=a$ for every $i\in I$
\end{definition}
In the following theorem (1) holds for cylindric algebras, Pinter's substitution algebras
(which are replacement algebras endowed with cylindrifiers) and quasipolydic algebras
with and without equality when the dimension is $\leq 2$. (3) holds for such algebras for all finite dimensions. (4) is due to Johnsson and Tarski.
In fact, (1) holds for any discriminator variety $V$ of $BAO$'s, with finitely many operators, when $V$ is generated by a discriminator class $K$.

\begin{theorem} Let $K$ be a variety of Boolean algebras with finitely many operators.
\begin{enumarab}
\item Assume that  $K=V(Fin(K))$, and for any $\B\in K$ and $b'\in \B$, there exists $b\in \B$ such that
$\Ig^{\B}\{b'\}=\Ig^{\Bl\B}\{b\}$. If $\A$ is finitely generated, then $\A$ is atomic.
In particular, the finitely generated free algebras are atomic.
\item Assume the condition above on principal ideals, together with the condition that
that if $b_1'$ and $b_2$'s are the generators of two given ideals happen to be disjoint then $b_0, b_1$ can be chosen to be also disjoint. Then
$\Fr_{\beta}K_{\alpha}\times \Fr_{\beta}K_{\alpha}\cong \Fr_{|\beta+1|}K.$ In particular if $\beta$ is infinite, and
$\A=\Fr_{\beta}K$, then $\A\times \A\cong \A$. (This happens when $b_0, b_1$ are closed.)
\item Assume that $\beta<\omega$, and assume the above condition on principal ideals.
Suppose further that for every $k\in \omega$, there exists an algebra $\A\in K$ with $k$ elements
that is generated by a single element. Then $\Fr_{\beta}K$ has infinitely many atoms.
\item  Assume that $K=V(Fin(K))$.
Suppose $\A$ is $K$ freely generated by a finite set $X$ and $\A=\Sg Y$ with $|Y|=|X|$. Then $\A$ is $K$ freely generated
by $Y.$
\end{enumarab}
\end{theorem}
\begin{proof}
\begin{enumarab}
\item Assume that $a\in A$ is non-zero. Let $h:\A\to \B$ be a homomorphism of $\A$ into a finite algebra $\B$ such that
$h(a)\neq 0$. Let $I=ker h.$ We claim that $I$ is a finitely generated ideal.
Let $R_I$ be the congruence relation corresponding to $I$, that is $R_I=\{(a,b)\in A\times A: h(a)=h(b)\}$.
Let $X$ be a finite set such that $X$ generates $\A$ and $h(X)=\B$. Such a set obviously exists.
Let $X'=X\cup \{x+y: x, y\in X\}\cup \{-x: x\in X\}\cup \bigcup_{f\in t}\{f(x): x\in X\}.$
Let $R=\Sg^{\A}(R_I\cap X\times X')$. Clearly $R$ is a finitely generated congruence and $R_I\subseteq R$.
We show that the converse inclusion also holds.
For this purpose we first show that $R(X)=\{a\in A: \exists x\in X (x,a)\in R\}=\A.$
Assume that $xRa$ and $yRb$, $x,y\in X$ then $x+yRa+b$, but there exists $z\in X$ such that $h(z)=h(x+y)$ and $zR(x+y)$, hence
$zR(a+b)$ , so that $a+b\in R(X)$. Similarly for all other operations. Thus $R(X)=A$.
Now assume that $a,b\in A$ such that $h(a)=h(b)$.
Then there exist $x, y\in X$ such that $xRa$ and $xRb$. Since $R\subseteq ker h$,
we have $h(x)=h(a)=h(b)=h(y)$ and so $xRy$, hence $aRb$ and $R_I\subseteq R$.
So $I=\Ig\{b'\}$ for some element $b'$.  Then there exists $b\in \A$ such that  $\Ig^{\Bl\B}\{b\}=\Ig\{b'\}.$ Since $h(b)=0$ and $h(a)\neq 0,$
we have $a.-b\neq 0$.
Now $h(\A)\cong \A/\Ig^{\Bl\B}\{b\}$ as $K$ algebras. Let $\Rl_{-b}\A=\{x: x\leq -b\}$. Let $f:\A/\Ig^{\Bl\B}\{b\}\to \Rl_{-b}\A$ be defined by
$\bar{x}\mapsto x.-b$. Then $f$ is an isomorphism of Boolean algebras (recall that the operations of $\Rl_{-b}\B$ are defined by
relativizing the Boolean operations to $-b$.)
Indeed, the map is well defined, by noting that if $x\delta y\in \Ig^{\Bl\B}\{b\}$, where $\delta$ denotes symmetric difference,
then $x.-b=y.-b$ because $x, y\leq b$.
Since $\Rl_{-b}\A$ is finite, and $a.-b\in \Rl_{-b}\A$ is non-zero, then there exists an atom $x\in \Rl_{-b}\A$ below $a$,
but clearly $\At(\Rl_{-b}\A)\subseteq \At\A$ and we are done.

\item Let $(g_i:i\in \beta+1)$ be the free generators of $\A=\Fr_{\beta+1}K$.
We first show that $\Rl_{g_{\beta}}\A$ is freely generated by
$\{g_i.g_{\beta}:i<\beta\}$. Let $\B$ be in $K$ and $y\in {}^{\beta}\B$.
Then there exists a homomorphism $f:\A\to \B$ such that $f(g_i)=y_i$ for all $i<\beta$ and $f(g_{\beta})=1$.
Then $f\upharpoonright \Rl_{g_{\beta}}\A$ is a homomorphism such that $f(g_i.g_{\beta})=y_i$. Similarly
$\Rl_{-g_{\beta}}\A$ is freely generated by $\{g_i.-g_{\beta}:i<\beta\}$.
Let $\B_0=\Rl_{g_{\beta}}\A$ and $\B_1=\Rl_{g_{\beta}}\A$.
Let $t_0=g_{\beta}$ and $t_1=-g_{\beta}$. Let $x_i$ be such that $J_i=\Ig\{t_i\}=\Ig^{Bl\A}\{x_i\}$, and $x_0.x_1=0$.
Exist by assumption. Assume that $z\in J_0\cap J_1$. Then $z\leq x_i$,
for $i=0, 1$, and so  $z=0$. Thus $J_0\cap J_1=\{0\}$. Let $y\in A\times A$, and let $z=(y_0.x_0+y_1.x_1)$, then $y_i.x_i=z.x_i$ for each $i=\{0,1\}$
and so $z\in \bigcap y_0/J_0\cap y_1/J_1$. Thus $\A/J_i\cong \B_i$, and so
$\A\cong \B_0\times \B_1$.

\item Let $\A=\Fr_{\beta}K.$ Let $\B$ have $k$ atoms and generated by a single element. Then there exists a surjective
homomorphism $h:\A\to \B$. Then, as in the first item,  $\A/\Ig^{\Bl\B}\{b\}\cong \B$, and so $\Rl_{b}\B$ has $k$ atoms.
Hence $\A$ has $k$ atoms for any $k$ and we are done.

\item Let $\A=\Fr_XK$, let $\B\in Fin(K)$ and let $f:X\to \B$. Then $f$ can extended to a homomorphism $f':\A\to \B$.
Let $\bar{f}=f'\upharpoonright Y$. If $f, g\in {}^XB$ and $\bar{f}=\bar{g}$,
then $f'$ and $g'$ agree on a generating set $Y$, so $f'=g',$ hence $f=g$.
Therefore we obtain a one to one mapping from $^XB$ to $^YB$, but $|X|=|Y|,$
hence this map is surjective. In other words for each $h\in {}^YB,$ there exists a unique
$f\in {}^XB$ such that $\bar{f}=h$, then $f'$ with domain $\A$ extends $h.$
Since $\Fr_XK=\Fr_X(Fin(K))$ we are done.
\end{enumarab}
\end{proof}

For cylindric algebras, Pinter's algebras and quasipolyadic equality, though free algebras of $>2$ dimensions
contain infinitely many atoms, they are not
atomic.
\begin{definition} Let $K$ be a class of $BAO$ with operators$(f_i: i\in I.)$
Let $\A\in K$. An element $b\in A$ is called hereditary closed if for all $x\leq b$, $f_i(x)=x$.
\end{definition}
\begin{theorem}
\begin{enumarab}
\item Let $\A=\Sg X$ and $|X|<\omega$. Let $b\in \A$ be hereditary closed. Then $\At\A\cap \Rl_{b}\A\leq 2^{n}$.
If $\A$ is freely generated by $X$, then $\At\A\cap \Rl_{b}\A= 2^{n}.$
\item If every atom of $\A$ is below $b,$ then $\A\cong \Rl_{b}\A\times \Rl_{-b}\A$, and $|\Rl_{b}\A|=2^{2^n}$.
If in addition $\A$ is infinite, then $\Rl_{-b}\A$ is atomless.
\end{enumarab}
\end{theorem}
\begin{proof} Assume that $|X|=m$. We have $|\At\A\cap \Rl_b\A|=|\{\prod Y \sim \sum(X\sim Y).b\}\sim \{0\}|\leq {}^{m}2.$
Let $\B=\Rl_b\A$. Then $\B=\Sg^{\B}\{x_i.b: i<m\}=\Sg^{Bl\B}\{x_i.b:i<\beta\}$ since $b$ is hereditary fixed.
For $\Gamma\subseteq m$, let
$$x_{\Gamma}=\prod_{i\in \Gamma}(x_i.b).\prod_{i\in m\sim \Gamma}(x_i.-b).$$
Let $\C$ be the two element algebra. Then for each $\Gamma\subseteq m$, there is a homomorphism $f:\A\to \C$ such that
$fx_i=1$ iff $i\in \Gamma$.This shows that $x_{\Gamma}\neq 0$ for every $\Gamma\subseteq m$,
while it is easily seen that $x_{\Gamma}$ and $x_{\Delta}$
are distinct for distinct $\Gamma, \Delta\subseteq m$. We show that $\A\cong \Rl_{b}\A\times \Rl_{-b}\A$.

Let $\B_0=\Rl_{b}\A$ and $\B_1=\Rl_{-b}\A$.
Let $t_0=b$ and $t_1=-b$. Let  $J_i=\Ig\{t_i\}$
Assume that $z\in J_0\cap J_1$. Then $z\leq x_i$,
for $i=0, 1$, and so  $z=0$. Thus $J_0\cap J_1=\{0\}$. Let $y\in A\times A$,
and let $z=(y_0.t_0+y_1.t_1)$, then $y_i.x_i=z.x_i$ for each $i=\{0,1\}$
and so $z\in \bigcap y_0/J_0\cap y_1/J_1$. Thus $\A/J_i\cong \B_i$, and so
$\A\cong \B_0\times \B_1$.

\end{proof}
The following theorem holds for any class of $BAO$'s.
\begin{theorem} The free algebra on an infinite generating set $X$ is atomless.
\end{theorem}
\begin{proof} Let $a\in A$ be non-zero. Then there is a finite $Y\subset X$ such that $a\in \Sg^{\A} Y$. Let $y\in X\sim Y$.
Then by freeness
there exist homomorphisms $f:\A\to \B$ and $h:\A\to \B$ such that $f(\mu)=h(\mu) $ for all $\mu\in Y$ while
$f(y)=1$ and $h(y)=0$. Then $f(a)=h(a)=a$. Hence $f(a.y)=h(a.-y)=a\neq 0$ and so $a.y\neq 0$ and
$a.-y\neq 0$.
Thus $a$ cannot be an atom.
\end{proof}
\subsection{Specific results on algebras of substitutions}

\begin{definition} A variety $V$ is locally finite, if every finitely generated algebra is finite.
\end{definition}
The most famous locally finite variety
is that of Boolean algebras. It turns out that our varieties are also locally finite,
as the next simple proof shows.


\begin{theorem}\label{locallyfinite}If $\A$ is generated by $X$ and $|X|=m$ then $|\A|\leq 2^{m\times n!}$
\end{theorem}
\begin{proof} We first prove the theorem when $\A$ is diagonal free.
Let $Y=\{s_{\tau}x_i: i<m, \tau\in S_n\}$.  Then
$\A=\Sg^{\Bl\A}Y$. This follows from the fact that the substitutions are Boolean endomorphisms.
Since $|Y|\leq m\times n!,$ then $|\A|\leq 2^{m\times n!}$.
When we have diagonal elements we take the larger, but still finite $Y=\{s_{\tau}x_i: i<m, \tau\in S_n\}\cup \{d_{ij}: i,j\in n\}$.
\end{proof}
\begin{theorem} Let $\A=\Fr_XV$. Assume that $X=\{g_i:i\in m\}$. Let $Y=\{s_{\tau}x: , \tau\in S_n: x\in X\}$ and $\beta=|Y|.$ Then $\beta=n^2\sim n$ and
there exists $i\in \omega$ such that $m\leq i\leq n$ and  $|\At \A|= 2^i$.
\end{theorem}
\begin{proof} We have $\A=\Sg^{\Bl A}Y$. Hence $\At \A=\{\prod Z\sim \sum (Y\sim Z): Z\subseteq Y\}\sim \{0\}$, and so
$|At \A|\leq 2^{\beta}.$  Conversely,  for each $\Gamma\subseteq m$,
let $x_{\Gamma}=\prod_{\eta\in \Gamma}g_{\eta}.\prod_{\eta\in \beta \sim \Gamma}-g_{\eta}$. Let $\C$ be the two element
substitution algebra. Then for each $\Gamma\subseteq m$, there is a homomorphism $f:\A\to \C$ such that $f(s_{[i,j]}g_{\eta})=1$ iff
$\eta\in \Gamma$and $i,j\in n$, and  hence $f(x_{\Gamma})=1$.  This shows that $x_{\Gamma}\neq 0$ for every $\Gamma\subseteq m$,
while it is easily seen that $x_{\Gamma}$ and $x_{\Delta}$
are distinct for distinct $\Gamma, \Delta\subseteq \beta$.
Hence $|\At\A|\geq 2^{m}$.
\end{proof}
One can prove easily show  that if $\alpha$ or $\beta$ are infinite, then $|\Fr_{\beta}SA_{\alpha}|=|\alpha|\cup |\beta|$.

\begin{question} Does $\Fr_{\beta}SA_n$, for finite $\beta$ and any ordinal $n$ have exactly $2^{\beta}$ atoms?
\end{question}

\section{Amalgamation}

\subsection{General theorems}

In this section we show that all varieties considered have the superamalgamation property, a strong form of amalgamation.
Like in the previous section, we start by formulating slightly new results in the general
setting of $BAO$'s. The main novelty in our general approach is that we consider classes
of algebras that are not necessarily varieties. We start by some definitions.

\begin{definition}
\begin{enumarab}
\item $K$ has the \emph{Amalgamation Property } if for all $\A_1, \A_2\in K$ and monomorphisms
$i_1:\A_0\to \A_1,$ $i_2:\A_0\to \A_2$
there exist $\D\in K$
and monomorphisms $m_1:\A_1\to \D$ and $m_2:\A_2\to \D$ such that $m_1\circ i_1=m_2\circ i_2$.
\item If in addition, $(\forall x\in A_j)(\forall y\in A_k)
(m_j(x)\leq m_k(y)\implies (\exists z\in A_0)(x\leq i_j(z)\land i_k(z) \leq y))$
where $\{j,k\}=\{1,2\}$, then we say that $K$ has the superamalgamation property $(SUPAP)$.
\end{enumarab}
\end{definition}

\begin{definition} An algebra $\A$ has the strong interpolation theorem, $SIP$ for short, if for all $X_1, X_2\subseteq A$, $a\in \Sg^{\A}X_1$,
$c\in \Sg^{\A}X_2$ with $a\leq c$, there exist $b\in \Sg^{\A}(X_1\cap X_2)$ such that $a\leq b\leq c$.
\end{definition}

For an algebra $\A$, $Co\A$ denotes the set of congruences on $\A$.
\begin{definition}

An algebra $\A$ has the congruence extension property, or $CP$ for short,
 if for any $X_1, X_2\subset A$
if $R\in Co \Sg^{\A}X_1$ and $S\in Co \Sg^{A}X_2$ and
$$R\cap {}^2\Sg^{A}(X_1\cap X_2)=S\cap {}^2\Sg^{\A}{(X_1\cap X_2)},$$
then there exists a congruence $T$ on $\A$ such that
$$T\cap {}^2 \Sg^{\A}X_1=R \text { and } T\cap {}^2\Sg^{\A}{(X_2)}=S.$$

\end{definition}

Maksimova and Madarasz proved that interpolation in free algebras of a variety imply that the variety has the
superamalgamation property.
Using a similar argument, we prove this implication in a slightly more
general setting. But first an easy  lemma:

\begin{lemma} Let $K$ be a class of $BAO$'s. Let $\A, \B\in K$ with $\B\subseteq \A$. Let $M$ be an ideal of $\B$. We then have:
\begin{enumarab}
\item $\Ig^{\A}M=\{x\in A: x\leq b \text { for some $b\in M$}\}$
\item $M=\Ig^{\A}M\cap \B$
\item if $\C\subseteq \A$ and $N$ is an ideal of $\C$, then
$\Ig^{\A}(M\cup N)=\{x\in A: x\leq b+c\ \text { for some $b\in M$ and $c\in N$}\}$
\item For every ideal $N$ of $\A$ such that $N\cap B\subseteq M$, there is an ideal $N'$ in $\A$
such that $N\subseteq N'$ and $N'\cap B=M$. Furthermore, if $M$ is a maximal ideal of $\B$, then $N'$ can be taken to be a maximal ideal of $\A$.

\end{enumarab}
\end{lemma}
\begin{demo}{Proof} Only (iv) deserves attention. The special case when $n=\{0\}$ is straightforward.
The general case follows from this one, by considering
$\A/N$, $\B/(N\cap \B)$ and $M/(N\cap \B)$, in place of $\A$, $\B$ and $M$ respectively.
\end{demo}
The previous lemma will be frequently used without being explicitly mentioned.

\begin{theorem} Let $K$ be a class of $BAO$'s such that $\mathbf{H}K=\mathbf{S}K=K$.
Assume that for all $\A, \B, \C\in K$, inclusions
$m:\C\to \A$, $n:\C\to \B$, there exist $\D$ with $SIP$ and $h:\D\to \C$, $h_1:\D\to \A$, $h_2:\D\to \B$
such that for $x\in h^{-1}(\C)$,
$$h_1(x)=m\circ h(x)=n\circ h(x)=h_2(x).$$
Then $K$ has $SUPAP$.
\bigskip
\bigskip
\bigskip
\bigskip
\bigskip
\bigskip
\bigskip
\bigskip
\bigskip
\bigskip
\bigskip
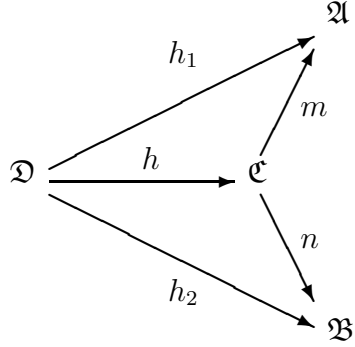
\begin{picture}(10,0)(-30,70)
\thicklines
\put (-10,0){$\D$}
\put(5,0){\vector(1,0){70}}\put(80,0){$\C$}
\put(5,5){\vector(2,1){100}}\put(110,60){$\A$}
\put(5,-5){\vector(2,-1){100}}\put(110,-60){$\B$}
\put(85,10){\vector(1,2){20}}
\put(85,-5){\vector(1,-2){20}}
\put(40,5){$h$}
\put(100,25){$m$}
\put(100,-25){$n$}
\put(50,45){$h_1$}
\put(50,-45){$h_2$}
\end{picture}
\end{theorem}
\bigskip
\bigskip
\begin{proof} Let $\D_1=h_1^{-1}(\A)$ and $\D_2=h_2^{-1}(\B)$. Then $h_1:\D_1\to \A$, and $h_2:\D_2\to \B$.

Let $M=ker h_1$ and $N=ker h_2$, and let
$\bar{h_1}:\D_1/M\to \A, \bar{h_2}:\D_2/N\to \B$ be the induced isomorphisms.

Let $l_1:h^{-1}(\C)/h^{-1}(\C)\cap M\to \C$ be defined via $\bar{x}\to h(x)$, and
$l_2:h^{-1}(\C)/h^{-1}(\C)\cap N$ to $\C$ be defined via $\bar{x}\to h(x)$.
Then those are well defined, and hence
$k^{-1}(\C)\cap M=h^{-1}(\C)\cap N$.
Then we show that $\P=\Ig(M\cup N)$ is a proper ideal and $\D/\P$ is the desired algebra.
Now let $x\in \mathfrak{Ig}(M\cup N)\cap \D_1$.
Then there exist $b\in M$ and $c\in N$ such that $x\leq b+c$. Thus $x-b\leq c$.
But $x-b\in \D_1$ and $c\in \D_2$, it follows that there exists an interpolant
$d\in \D_1\cap \D_2$  such that $x-b\leq d\leq c$. We have $d\in N$
therefore $d\in M$, and since $x\leq d+b$, therefore $x\in M$.
It follows that
$\mathfrak{Ig}(M\cup N)\cap \D_1=M$
and similarly
$\mathfrak{Ig}(M\cup N)\cap \D_2=N$.
In particular $P=\mathfrak{Ig}(M\cup N)$ is a proper ideal.

Let $k:\D_1/M\to \D/P$ be defined by $k(a/M)=a/P$
and $h:\D_2/N\to \D/P$ by $h(a/N)=a/P$. Then
$k\circ m$ and $h\circ n$ are one to one and
$k\circ m \circ f=h\circ n\circ g$.
We now prove that $\D/P$ is actually a
superamalgam. i.e we prove that $K$ has the superamalgamation
property. Assume that $k\circ m(a)\leq h\circ n(b)$. There exists
$x\in \D_1$ such that $x/P=k(m(a))$ and $m(a)=x/M$. Also there
exists $z\in \D_2$ such that $z/P=h(n(b))$ and $n(b)=z/N$. Now
$x/P\leq z/P$ hence $x-z\in P$. Therefore  there is an $r\in M$ and
an $s\in N$ such that $x-r\leq z+s$. Now $x-r\in \D_1$ and $z+s\in\D_2,$ it follows that there is an interpolant
$u\in \D_1\cap \D_2$ such that $x-r\leq u\leq z+s$. Let $t\in \C$ such that $m\circ
f(t)=u/M$ and $n\circ g(t)=u/N.$ We have  $x/P\leq u/P\leq z/P$. Now
$m(f(t))=u/M\geq x/M=m(a).$ Thus $f(t)\geq a$. Similarly
$n(g(t))=u/N\leq z/N=n(b)$, hence $g(t)\leq b$. By total symmetry,
we are done.
\end{proof}

The intimate relationship between $CP$ and $AP$ has been worked out extensively by Pigozzi
for cylindric algebras. Here we prove an implication in one direction for $BAO$'s.

\begin{theorem}
Let $K$ be such that $\mathbf{H}K=\mathbf{S}K=K$. If $K$ has the amalgamation  property, then the $V(K)$
free algebras have $CP$.
\end{theorem}

\begin{proof}
For $R\in Co\A$ and $X\subseteq  A$, by $(\A/R)^{(X)}$ we understand the subalgebra of
$\A/R$ generated by $\{x/R: x\in X\}.$ Let $\A$, $X_1$, $X_2$, $R$ and $S$ be as specified in in the definition of $CP$.
Define $$\theta: \Sg^{\A}(X_1\cap X_2)\to \Sg^{\A}(X_1)/R$$
by $$a\mapsto a/R.$$
Then $ker\theta=R\cap {}^2\Sg^{\A}(X_1\cap X_2)$ and $Im\theta=(\Sg^{\A}(X_1)/R)^{(X_1\cap X_2)}$.
It follows that $$\bar{\theta}:\Sg^{\A}(X_1\cap X_2)/R\cap {}^2\Sg^{\A}(X_1\cap X_2)\to (\Sg^{\A}(X_1)/R)^{(X_1\cap X_2)}$$
defined by
$$a/R\cap {}^{2}\Sg^{\A}(X_1\cap X_2)\mapsto a/R$$
is a well defined isomorphism.
Similarly
$$\bar{\psi}:\Sg^{\A}(X_1\cap X_2)/S\cap {}^2\Sg^{\A}(X_1\cap X_2)\to (\Sg^{\A}(X_2)/S)^{(X_1\cap X_2)}$$
defined by
$$a/S\cap {}^{2}\Sg^{\A}(X_1\cap X_2)\mapsto a/S$$
is also a well defined isomorphism.
But $$R\cap {}^2\Sg^{\A}(X_1\cap X_2)=S\cap {}^2\Sg^{\A}(X_1\cap X_2),$$
Hence
$$\phi: (\Sg^{\A}(X_1)/R)^{(X_1\cap X_2)}\to (\Sg^{\A}(X_2)/S)^{(X_1\cap X_2)}$$
defined by
$$a/R\mapsto a/S$$
is a well defined isomorphism.
Now
$(\Sg^{\A}(X_1)/R)^{(X_1\cap X_2)}$ embeds into $\Sg^{\A}(X_1)/R$ via the inclusion map; it also embeds in $\A^{(X_2)}/S$ via $i\circ \phi$ where $i$
is also the inclusion map.
For brevity let $\A_0=(\Sg^{\A}(X_1)/R)^{(X_1\cap X_2)}$, $\A_1=\Sg^{\A}(X_1)/R$ and $\A_2=\Sg^{\A}(X_2)/S$ and $j=i\circ \phi$.
Then $\A_0$ embeds in $\A_1$ and $\A_2$ via $i$ and $j$ respectively.
Then there exists $\B\in V$ and monomorphisms $f$ and $g$ from $\A_1$ and $\A_2$ respectively to
$\B$ such that
$f\circ i=g\circ j$.
Let $$\bar{f}:\Sg^{\A}(X_1)\to \B$$ be defined by $$a\mapsto f(a/R)$$ and $$\bar{g}:\Sg^{\A}(X_2)\to \B$$
be defined by $$a\mapsto g(a/R).$$
Let $\B'$ be the algebra generated by $Imf\cup Im g$.
Then $\bar{f}\cup \bar{g}\upharpoonright X_1\cup X_2\to \B'$ is a function since $\bar{f}$ and $\bar{g}$ coincide on $X_1\cap X_2$.
By freeness of $\A$, there exists $h:\A\to \B'$ such that $h\upharpoonright_{X_1\cup X_2}=\bar{f}\cup \bar{g}$.
Let $T=kerh $. Then it is not hard to check that
$$T\cap {}^2 \Sg^{\A}(X_1)=R \text { and } T\cap {}^2\Sg^{\A}(X_2)=S.$$
\end{proof}
Finally we show that $CP$ implies a weak form of interpolation.
\begin{theorem}
If an algebra $\A$ has $CP$ , then for $X_1, X_2\subseteq \A$, if $x\in \Sg^{\A}X_1$ and $z\in \Sg^{\A}X_2$ are such that
$x\leq z$, then there exists $y\in \Sg^{\A}(X_1\cap X_2),$ and a term $\tau$ such that $x\leq y\leq \tau(z)$.
If $Ig^{\Bl\A}\{z\}=\Ig^{\A}\{z\},$ then $\tau$ can be chosen to be the
identity term. In particular, if $z$ is closed then the latter case occurs.
\end{theorem}

\begin{proof}
Now let $x\in \Sg^{\A}(X_1)$, $z\in \Sg^{\A}(X_2)$ and assume that $x\leq z$.
Then $$x\in (\Ig^{\A}\{z\})\cap \Sg^{\A}(X_1).$$
Let $$M=\Ig^{\A^{(X_1)}}\{z\}\text { and } N=\Ig^{\Sg^{\A}(X_2)}(M\cap \Sg^{\A}(X_1\cap X_2)).$$
Then $$M\cap \Sg^{\A}(X_1\cap X_2)=N\cap \Sg^{\A}(X_1\cap X_2).$$
By identifying ideals with congruences, and using the congruence extension property,
there is a an ideal $P$ of $\A$
such that $$P\cap \Sg^{\A}(X_1)=N\text { and }P\cap \Sg^{\A}(X_2)=M.$$
It follows that
$$\Ig^{\A}(N\cup M)\cap \Sg^{\A}(X_1)\subseteq P\cap \Sg^{\A}(X_1)=N.$$
Hence
$$(\Ig^{(\A)}\{z\})\cap A^{(X_1)}\subseteq N.$$
and we have
$$x\in \Ig^{\Sg^{\A}X_1}[\Ig^{\Sg^{\A}(X_2)}\{z\}\cap \Sg^{\A}(X_1\cap X_2).]$$
This implies that there is an element $y$ such that
$$x\leq y\in \Sg^{\A}(X_1\cap X_2)$$
and $y\in \Ig^{Sg^{\A}X}\{z\}$, hence the first required. The second required follows
follows, also immediately, since $y\leq z$, because $\Ig^{\A}\{z\}=\Rl_z\A$.
\end{proof}

\subsection{Specific theorems for algebras of substitutions}

By an algebra $\A$ we mean a substitution algebra.
For an algebra $\A$ and $X\subseteq  A$, $fl^{\A}X$
denotes the Boolean filter generated by $X$.

\begin{theorem} let $\A=\Fr_XV$, and let $X_1, X_2\subseteq \A$ be such that $X_1\cup X_2=X$.
Assume that $a\in \Sg^{\A}X_1$ and $c\in \Sg^{\A}X_2$ are such that $a\leq c$.
Then there exists an interpolant $b\in \Sg^{\A}(X_1\cap X_2)$ such that
$a\leq b\leq c$.
\end{theorem}
\begin{proof} We prove the theorem for the finite dimensional case and for $S_n$. All other cases, for finite as well as for infinite dimensions,
can be accomplished in exactly the same manner,
undergoing the obvious modifications. In case, we have diagonals, we just factor out the base of the representations constructed as
in the previous proofs.

Assume that $a\leq c$, but there is no such $b$. We will reach a contradiction.
Let $$H_1=fl^{\Bl\Sg^{\A}X_1}\{a\}=\{x: x\geq a\},$$
$$H_2=fl^{\Bl\Sg^{\A}X_2}\{-c\}=\{x: x\geq -c\},$$
and $$H=fl^{\Bl\Sg^{\A}(X_1\cap X_2)}[(H_1\cap \Sg^{\A}(X_1\cap X_2))\cup (H_2\cap \Sg^{\A}(X_1\cap X_2))].$$
We show that $H$ is a proper filter of $\Sg^{\A}(X_1\cap X_2)$.
For this, it suffices to show that for any $b_0,b_1\in \Sg^{\A}(X_1\cap X_2)$, for any $x_1\in H_1$ and $x_2\in H_2$
 if $a.x_1\leq b_0$ and $-c.x_2\leq b_1$, then $b_0.b_1\neq 0$.
Now $a.x_1=a$ and $-c.x_2=-c$. So assume, to the contrary, that $b_0.b_1=0$. Then $a\leq b_0$ and $-c\leq b_1$ and so
$a\leq b_0\leq-b_1\leq c$, which is impossible because we assumed that there is no interpolant.

Hence $H$ is a proper filter. Let $H^*$ be an ultrafilter of $\Sg^{\A}(X_1\cap X_2)$ containing $H$, and let $F$ be an ultrafilter of $\Sg^{\A}X_1$
and $G$ be an ultrafilter of $\Sg^{\A}X_2$ such that $$F\cap \Sg^{\A}(X_1\cap X_2))=H^*=G\cap \Sg^{\A}(X_1\cap x_2).$$
Such ultrafilters exist.

For simplicity of notation let $\A_1=\Sg^{\A}(X_1)$ and $\A_2=\Sg^{\A}(X_2).$
Define $h_1:\A_1\to \wp(S_n)$ by
$$h_1(x)=\{\eta\in S_n: x\in s_{\eta}F\},$$
and
$h_2:\A_1\to \wp(S_n)$ by
$$h_2(x)=\{\eta\in S_n: x\in s_{\eta}G_\},$$
Then $h_1, h_2$
are homomorphisms, they agree on $\Sg^{\A}(X_1\cap X_2).$
Indeed let $x\in \Sg^{\A}(X_1\cap X_2)$. Then $\eta\in h_1(x)$ iff $x\in F_{\eta}$ iff $s_{\eta}x\in F$ iff
$s_{\eta}x\in F\cap \Sg^{\A}(X_1\cap X_2)=H^*=G\cap \Sg^{\A}(X_1\cap X_2)$ iff $s_{\eta}x\in G$ iff $x\in G_{\eta}$ iff $\eta\in h_2(x)$.
Thus  $h_1\cup h_2$ is a function. By freeness
there is an $h:\A\to \wp(S_n)$ extending $h_1$ and $h_2$. Now $Id\in h(a)\cap h(-c)\neq \emptyset$ which contradicts
$a\leq c$.

\end{proof}

\begin{corollary}\label{SUPAP} All varieties considered have the superamalgamation property.
\end{corollary}

\begin{proof}Let $\A, \B, \C\in K$, inclusions
$m:\C\to \A$, $n:\C\to \B$ be given. Take $\D$ to be the free algebra on a set $I\cup J$ of generators such that
$|I|=|A|$, $|J|=|B|$ and $|I\cap J|=|C|$. Then clearly there exist $h:\D\to \C$, $h_1:\D\to \A$, $h_2:\D\to \B$
such that for $x\in h^{-1}(C)$ ,
$h_1(x)=m\circ h(x)=n\circ h(x)=h_2(x).$
\end{proof}

\subsection*{Remark}

Like representability the infinite dimensional case may be inferred from the finite dimensional case as follows.
Let $\tau$ and $\sigma$ be terms in the language of $SA_{\alpha}$ and assume that $K\models \tau \leq \sigma$.
Then there is a finite $n$ such that $K_n\models \tau\leq \sigma$ and a we can find an interpolant.

\subsubsection{Another proof}

Here we give a different syntactical proof, depending on the fact that our varieties can be axiomatized by a set of positive equations.
This follows from the simple observation that Boolean algebras can be defined by equations involving only meet and join,
and so Boolean homomorphisms can be defined to respect only those two operations,
so that we can get rid of any reference to negation in our axioms. We prove our theorem only for transposition algebras, the rest of the cases
are the same.

\begin{definition}
\begin{enumarab}

\item A frame of type $TA_{\alpha}$ is a first order structure $\F=(V,  S_{ij})_{i,j\in \alpha}$ where $V$ is an arbitrary set and
and  $S_{ij}$ is a binary relation on $V$  for all $i, j\in \alpha$.

\item Given a frame $\F$, its complex algebra denote by $\F^+$ is the algebra $(\wp(\F), s_{ij})_{i,j}$  where for $X\subseteq  V$,
$s_{ij}(X)=\{s\in V: \exists t\in X, (t, s)\in S_{i,j} \}$.

\item Given $K\subseteq TA_{\alpha},$ then $\Cm^{-1}K=\{\F: \F^+\in K\}.$

\item Given a family $(\F_i)_{i\in I}$ of frames, a zigzag product of these frames is a substructure of $\prod_{i\in I}\F_i$ such that the
projection maps restricted to $S$ are
onto.
\end{enumarab}

\end{definition}

\begin{theorem}(Marx)
Assume that $K$ is a canonical variety and $L=\Cm^{-1}K$ is closed under finite zigzag products. Then $K$ has the superamalgamation
property.
\end{theorem}

\begin{theorem}
The variety $TA_{\alpha}$ has $SUPAP$.
\end{theorem}
\begin{proof}  Since $TA_{\alpha}$ is defined by positive equations then it is canonical.
In this case $L=\Cm^{-1}TA_{\alpha}$ consists of frames $(V, S_{i,j})$
such that if $s\in V$, then $s\circ [i,j]\in V$ and $s\circ [i,j]$ is in $V$.
The first order correspondents of the positive equations translated to the class of frames will be Horn formulas, hence clausifiable
and so $L$ is closed under finite zigzag products. Marx's theorem finishes the proof.
\end{proof}

\subsection*{Remark}

When we add cylindrifications things blow up. For such algebras the class of subdirect products of set algebras is a variety,
but it is not finitely axiomatizable not decidable, and the free algebras are not atomic.
If we do not insist on commutativity of cylindrifiers we get a nice representation theory, witness the theorems of Andreka Resek Thompson
Ferenzci [reference to be provided].

\section{Logical consequences}

Let $\L_n$ denote the fragment of first order logic with $n$ many variables.
A particular language has countably many relation symbols of the form
$R(x_0,..x_{n-1})$, where the variables occur n their natural order and and the $s_{[i,j]}$'s are treated as connectives. A structure $\M$ is specified,
like ordinary first order logic, specifying for every relation symbol an
$n$-ary relation on $M$ the domain of $\M$. Satisfiability is defined inductively the usual way:
For $s\in {}^nM$ and a formula $\phi$, $s\in {}^nM$, $\M\models s_{[i,j]}\phi$ iff
$s\circ [i,j]$ satisfies $\phi$.
For a structure $\M$, a formula $\phi$,  and an assignment $s\in {}^nM$, we write
$\M\models \phi[s]$ if $s$ satisfies $\phi$ in $\M$. We write $\phi^{\M}$  for the set of all assignments satisfying $\phi.$
Then the algebra with universe $\{\phi^{\M}:\phi\in \L\}$ is a set algebra.

Now consider the basic declarative statement in this fragment of first order logic
concerning the truth of a formula in a model under an assignment $s$.
$\M\models \phi[s].$
We can read this from a modal perspective `the formula $\phi$ is true in $\M$ at state $s$'.
Indeed we can replace the above truth definition
with the modal equivalent
$\M\models s_{[i,j]}\phi[s]$ iff there is  $t\in M$ with $t\equiv_{i,j}s$ and $\M\models \phi[t]$
where $\equiv_{i,j}$ is the relation on $^nM$ defined by $s\equiv_{i,j} t$ iff $s\circ [i,j]=t.$
In other words, the substitutions behave like a modal diamond having
$\equiv{i,j}$ as its accessibility relation.
So we can look at set algebras as complex algebras  of frames of the form $(^nU, \equiv_{i,j}).$
Since the semantics of the Boolean connectives in the predicate calculus is the same as in modal logic,
this shows that the inductive clauses in the truth definition of first order logic
neatly fit a modal mould. In fact, the modal disguise of this fragment of first order logic  is so thin,
that there is an absolutely straightforward translation mapping formulas to modal ones.
But we can relativize the set of states to permutable sets of sequences.

Now, bearing this double view in mind,  we consider the multi dimensional modal logic $\L$ corresponding to $TA_n$.
We chose to work with $S_n$ since it has more (metalogical) theorems. The infinite dimensional case is also identical modulo replacing $S_n$
by finite permutation in $^{\alpha}\alpha^{(Id)}$ where $\alpha$ is the dimension.
The proofs of the common theorems are the same underlying the obvious  modifications.

$\L$ has a set $P$ of countably many propositional variables, the Boolean connections and a modality $R_{i,j}$ for every
$i,j\in n$. A frame is a tuple $(V, R_{i,j})$ where $V$ is permutable; called the set of states and $R_{i,j}$ are binary relations on $V$ defined by
$(s,t)\in R_{i,j}$ if $s\circ [i,j]=t$. A model $\M$ is a triple $(V, R_{i,j}, s)$ where $(V,R_{i,j})$ is a frame and $s:P\to \wp(V)$.
The notion of satisfiability in $\M$ of a formula $\phi$ at state $w$ is defined inductively the usual way,
and the semantical relation $\models$ defined accordingly.

Now terms in the language of $TA_n$ translates to formulas also the usual way. One translates effectively the set of axioms of
$TA_n$ to a finite set of formula schema  $Ax$
each of the form of an equivalence. This can be done inductively. For a term $t$ write $\phi_t$ for the corresponding formula schema.
Then we have $TA_n\models t_1=t_2$ iff $Ax\vdash \phi_{t_1}\leftrightarrow \phi_{t_2}$.
We now formulate the metalogical counterparts of our algebraic results using standard machinery of algebraic  logic.

\begin{theorem} $Ax$ with modus ponens  is a finite complete Hilbert-style axiomatization. That is for any set $\Gamma$ of formulas
$\Gamma\models \phi,$ then $\Gamma\vdash \phi$. Furthermore, there is an effective proof of $\phi$.
\end{theorem}
\begin{proof} We prove that any consistent set $T$ of formula is satisfiable, and indeed satisfiable in a finite model.
Assume that $T$ and $\phi$ are given. Form the Lindenbaum Tarski algebra $\A=\Fm_T$ and let $a=\phi/T$.
We have $a$ is non-zero, because $\phi$ is consistent with $T$.
Let $\B$ be a set algebra with unit $D$ and $f:\A\to \wp(D)$ be a representation
such that $f(a)\neq 0$.
We extract a model $D$ of $T$, with base $M$, from $\B$ as follows.
For  a relation symbol $R$ and $s\in D$, $D,s\models R$ if $s\in f(R(x_0,x_1\ldots ..)/T)$. Here the variables occur in their
natural order.
\end{proof}
\begin{corollary}$\L$ has the finite base property that is if $\phi$ is satisfiable in a model, then it is satisfiable in a finite model
$\L$ is complete with respect to the class of finite frames.
\end{corollary}
\begin{proof} Since the variety considered is locally finite.
\end{proof}
\begin{theorem}
\begin{enumarab}
\item $\L$ has the Craig interpolation property. That is to say, if $\phi, \psi$ are formulas such that $\phi\to \psi$ then there is a formula
$\theta$ in their common vocabulary such that $\vdash \phi\to \theta$ and $\theta\to \phi.$
\item $\L$ has the joint consistency property
that if $T_1$ and $T_2$ are theories such that $T_1\cap T_2$ is consistent then $T_1\cup T_2$ is consistent.
\item $\L$ has the Beth definability property
\end{enumarab}
\end{theorem}
\begin{proof} By theorem \ref{SUPAP}, by noting that $SUPAP$ implies $CP$, and implies that epimorphisms are surjective,
which is the equivalent of Beth definability.
\end{proof}
\begin{definition} Let $T$ be a theory. A set $\Gamma$ is principal if there exists $\phi$ consistent with $T$ such that $T\models \phi\to \Gamma$.
Else $\Gamma$ is non-principal.
\end{definition}

\begin{theorem} If $\Gamma$ is non-principal, then there is a model $D$ of $T$ for which there is no $w$ such that $\M,w\models  \phi$ for all $\phi$
in $\Gamma.$
\end{theorem}
\begin{proof} By theorem \ref{OTT}
\end{proof}
The above theorem extends to omitting $< covK$ types.
\begin{definition} Let $T$ be a given $\L$ theory.
\begin{enumarab}
\item A formula $\phi$ is said to be complete in $T$ iff for every formula $\psi$ exactly one of
$$T\models \phi\to \psi, \\ T\models \phi\to \neg \psi$$
holds.
\item A formula $\theta$ is completable in $T$ iff there is a complete formula $\phi$ with $T\models \phi\to \theta$.
\item $T$ is atomic iff if every formula consistent with $T$ is completable in $T.$

\item A model $D$ of $T$ is atomic iff for every $s\in V$, there is a complete formula $\phi$ such that $V, s\models \phi.$
\end{enumarab}
\end{definition}

\begin{theorem} If $T$ is atomic, then $T$ has a model $D$, such that for any state $w$ there is an atomic formula $\phi$ such that
$D,w\models \phi$.
\end{theorem}
\begin{proof} From theorem \ref{com}
\end{proof}

\end{document}